\documentclass[a4paper,12pt,reqno]{amsart}
\usepackage{pdfsync}
\usepackage[utf8]{inputenc}
\usepackage{amsmath,amssymb,amsfonts,amsbsy,setspace, amsthm,}
\usepackage{algorithm}
\usepackage{algpseudocode}
\usepackage{mathrsfs}
\usepackage{verbatim}
\usepackage{enumerate}
\usepackage{eucal}

\usepackage{anysize}
\marginsize{4cm}{4cm}{3cm}{3cm}
\usepackage{cite}
\usepackage{hyperref}
\usepackage{graphicx}
\usepackage{color}
\usepackage{caption}
\usepackage{subcaption}
\newtheorem{theorem}{Theorem}
\newtheorem{corollary}[theorem]{Corollary}
\newtheorem{lemma}[theorem]{Lemma}
\newtheorem{proposition}[theorem]{Proposition}
\newtheorem{remark}[theorem]{Remark}

\newcommand{\R}{\mathbb{R}}

\newcommand\norm[1]{\left\Vert#1\right\Vert}

\newcommand\dd{\mathrm{d}}

\def\<#1,#2>{\left<#1,#2\right>}
\let\bar\overline

\DeclareMathOperator{\diag}{diag}

\renewcommand\phi{\varphi}

\renewcommand\epsilon{\varepsilon}

\newcommand{\TabThree}[1]{ %
\begin{tabular}{@{}c@{\hspace{1mm}}c@{\hspace{1mm}}c@{}}
 #1
\end{tabular}
}

\title[An ODE characterisation of MMOT]{An ODE characterisation of multi-marginal optimal transport with pairwise cost functions}

\author[L. Nenna]{Luca Nenna}
\address{Universit\'e Paris-Saclay, CNRS, Laboratoire de math\'ematiques d'Orsay, 91405, Orsay, France. }
\email{luca.nenna@universite-paris-saclay.fr}

\author[B. Pass]{Brendan Pass}
\address{Department of Mathematical and Statistical Sciences, 632 CAB, University of Alberta, Edmonton, Alberta, Canada, T6G 2G1}
\email{pass@ualberta.ca}

\begin{document}

\maketitle
\begin{abstract}
    The purpose of this paper is to introduce a new numerical method to solve multi-marginal optimal transport problems with pairwise interaction costs.  The complexity of multi-marginal optimal transport generally scales exponentially in the number of marginals $m$.  We introduce a one parameter family of cost functions that interpolates between the original and a special cost function for which the problem's complexity scales linearly in $m$.  We then show that the solution to the original problem can be recovered by solving an ordinary differential equation in the parameter $\epsilon$, whose initial condition corresponds to the solution for the special cost function mentioned above; we then present some simulations, using both explicit Euler and explicit higher order Runge-Kutta schemes to compute solutions to the ODE, and, as a result, the multi-marginal optimal transport problem.
\end{abstract}
\section{Introduction}
The theory of optimal transport plays an important role in many applications (see \cite{Villani-OptimalTransport-09,Villani-TOT2003,santambook,peyre2017computational}). Its generalization to the multi-marginal case consists in minimizing the functional
\[\gamma\mapsto\int_{X^1 \times ...\times X^m}c(x^1,\cdots,x^m)\dd\gamma \]
among all probability measures $\gamma\in\mathcal P(X^1\times\cdots\times X^m)$ having fixed $\mu^i\in\mathcal P(X^i)$ with $i=1,\cdots,m$ as marginals, for a given cost function $c(x^1,....,x^m)$.
This problem has been at the center of growing interest in recent years since it arises naturally in many different areas of applications,  including Economics \cite{carlier2010matching}, Financial Mathematics \cite{beiglbock2013model,dolinsky2014martingale,dolinsky2014robust, Ennajietal22}, Statistics \cite{bigot2018characterization,carlier2016vector}, Image Processing \cite{rabin2011wasserstein}, Tomography \cite{abraham2017tomographic}, Machine Learning \cite{Hassleretal21, Trillosetal22}, Fluid Dynamics \cite{brenier1989least} and Quantum Physics and Chemistry, in the framework of Density Functional Theory \cite{buttazzo2012optimal,cotar2013density}.

The structure of solutions to the multi-marginal optimal transport problem is a notoriously delicate issue, and is still not well understood, despite substantial efforts on the part of many researchers \cite{gangbo1998optimal, Carlier03, CarlierNazaret08, Heinich05, Pass11,Pass12, KimPass14, KimPass15, ColomboDePascaleDiMarino15, ColomboStra16, PassVargas21, MoameniPass17, PassVargas2022, PassVargas21-2}; see also the surveys \cite{PassSurvey} and \cite{DiMarinoGerolinNenna15}. In many of the aforementioned applications, it is therefore pertinent to develop efficient numerical algorithms to compute solutions.  This, however, represents a significant challenge,  since the problem  amounts to a linear (or convex, in a popular regularized variant discussed below), 
yet high dimensional optimization problem: the complexity scales exponentially in the number $m$ of marginals. For instance a crude discretization of each of $5$ marginals (notice that in many applications the number of marginals could be dramatically large, e.g. in quantum mechanics where $m$ is the number of electrons in a molecule) using 100 Dirac masses would mean that the coupling $\gamma$ between the $5$ marginals is supported over $100^5 = 10^{10}$ Dirac masses, rendering  the problem practically intractable.
There have been recently some attempts to tackle this problem by using different approaches: entropic regularization \cite{benamou2017generalized,thesislulu,benamou2016numerical}, relaxation via moment constraints approximation \cite{alfonsi2021approximation,alfonsi2022constrained}, genetic column generation algorithm exploiting the existence of a sparse solution in the discrete case   \cite{friesecke2022genetic,friesecke2022gencol},  Wasserstein penalisation of the marginal constraints \cite{merigot2016minimal} and semidefinite relaxation \cite{khoo2019convex,khoo2020semidefinite}.\\ 

 In many cases of interest, the cost function $c(x^1,....x^m) =\sum_{i,j=1, i \neq j}^mw(x^i,x^j)$ is given by a sum of two marginal cost functions; when $w(x^i,x^j) =|x^i-x^j|^2$, for instance the multi-marginal problem is equivalent to the well known Wasserstein barycenter problem (see Proposition 4.2 in \cite{Carlier_wasserstein_barycenter}), while the Coulomb cost $w(x^i,x^j) =\frac{1}{|x^i-x^j|}$ plays a central role in the quantum chemistry applications pioneered in \cite{cotar2013density} and \cite{buttazzo2012optimal}.   Here, for such pairwise interaction costs, our aim is to develop a continuation method which, by introducing a suitable one parameter family of cost functions,  establishes a link between the original multi-marginal problem and a simpler one whose complexity scales linearly in the number of marginals. For discrete marginals, we show that, after the addition of an entropic regularization term, the solution of the original multi-marginal problem can be recovered by solving an ordinary differential equation (ODE) whose initial condition is the solution to the simpler problem. 
This method  is actually inspired by the one introduced in \cite{CarlierGalichonSantambrogio10}  to compute the Monge solution of the two marginal problem, starting from the Knothe-Rosenblatt rearrangement;  note, however, note that since we apply this strategy to a regularized problem, our resulting ODE enjoys better regularity than the one in \cite{CarlierGalichonSantambrogio10}, which, in turn, makes it amenable to higher order numerical schemes (see the description of numerical results in Section \ref{algorithm} below).  
The above mentioned  differential equation will be derived  by differentiating the optimality conditions of the dual problem; in particular by penalizing the constraints with the soft-max function we will obtain a well defined ODE for which existence and uniqueness of a solution can be established.   

When developing the ODE approach in Section \ref{sect: ode approach} below, we restrict our attention to the case when the marginals $\mu^i$ are all identical. This has the significant advantage of reducing the Kantorovich dual problem to a maximization over a single potential function,  while also capturing important applications arising in density functional theory.  Though we do not pursue this direction here, our approach, with minor modifications, will also work with distinct marginals.  In this case, if each measure is discretized using $N$ points, one would need to solve $(m-1)N$ coupled, real-valued ODEs (rather than the $N$ coupled ODEs dealt with here  in Section \ref{sect: ode approach}), reflecting the $m-1$ independent Kantorovich potentials needed to fully capture the solution.  In this setting, again with minor modifications, the pairwise interactions may in fact differ as well; that is, we may consider $c(x^1,....x^m) =\sum_{i,j=1, i \neq j}^mw_{ij}(x^i,x^j)$ where the $w_{ij}$ differ for different pairs $i,j$.   Although we do not present the detailed calculations for these non symmetric problems (since they are mathematically very close to the symmetric case, though more notationally cumbersome), we do present some numerical simulations for one such problem (relating to Brenier's relaxation of the variational formulation of the incompressible Euler equation) to illustrate how the numerics can be extended. In addition, a separate manuscript on a related but somewhat different approach, which deals with completely general (not necessarily pairwise) cost functions is currently under preparation.


The remainder of this manuscript is organized as follows.  In Section 2, we recall some basic facts about multi-marginal optimal transport, as well as its entropic regularization and the duals of both problems, and prove that for a  particular, simple cost function, the solutions to the regularized problem and its dual can be computing by solving $m-1$ two marginal problems.  This solution will serve as the initial condition for an ODE, which is introduced, and proven to be well-posed, in Section 3.  In Section 4, algorithms, based on this ODE, to compute the solution to the multi-marginal optimal transport problem are described and some resulting numerical simulations are presented.

\section{Multi-marginal optimal transport and entropic regularization}

Given $m$ probability measures $\mu^i$ on bounded domains $X^i \subseteq \mathbb{R}^n$ for $i=1,2...,m$ and a lower semi-continuous cost function $c: X^1 \times X^2 \times ...\times X^m \rightarrow \mathbb{R} \cup \{+\infty\}$, the multi-marginal optimal transport problem consists in solving the following optimization problem

\begin{equation}
\label{pb:mmot}
\inf_{\gamma\in\Pi(\mu^1,\cdots,\mu^m)}  \int_{X^1 \times X^2...\times X^m} c(x^1,...,x^m)\dd\gamma
\end{equation}
where $\Pi(\mu^1,\cdots,\mu^m)$ denotes the set of probability measures on $X^1 \times X^2 \times ...\times X^m$ whose marginals are the $\mu^i$. One can easily show by means of the direct method of calculus of variations that this problem admits at least a solution, which will be referred as the \textit{optimal transport plan}.
It is well known that under some mild assumptions the above problem is dual to the following 
\begin{equation}\label{eqn: multi-marginal dual}
    \sup\left\{  \sum_{i=1}^m\int_{X^i}\phi^i(x^i)\dd\mu^i \;|\;\phi^i \in L^1(\mu^i),\text{ }\sum_{i=1}^m\phi^i(x^i) \leq c(x^1,...,x^m)\right \}.
\end{equation}
We will also consider a common variant of \eqref{pb:mmot}, known as \textit{entropic optimal transport} which consists in adding an entropy regularization term.  For a parameter $\eta >0$, this is to minimize
\begin{equation}
\label{pb:mmot-reg}
\inf_{\gamma\in\Pi(\mu^1,\cdots,\mu^m)}\left\{  \int_{X^1 \times X^2...\times X^m} c(x^1,....,x^m)\dd\gamma + \eta H_{\otimes_{i=1}^m\mu^i}(\gamma)\right \}
\end{equation}
where the relative entropy $H_{\otimes_{i=1}^m\mu^i}(\gamma)$ with respect to product measure $\otimes_{i=1}^m\mu^i$ is defined by 
$$
H_{\otimes_{i=1}^m\mu^i}(\gamma) = \int_{X^1 \times...\times X^m} \frac{\dd\gamma}{\dd(\otimes_{i=1}^m\mu^i)}\log(\frac{\dd\gamma}{\dd(\otimes_{i=1}^m\mu^i)}) \dd(\otimes_{i=1}^m\mu^i),
$$ 
if $\gamma$ is absolutely continuous with respect to the product measure $\otimes_{i=1}^m\mu^i$ and $+\infty$ otherwise.  The regularized transport is, in turn, dual to the following unconstrained optimization problem 
\begin{equation}\label{eqn: multi-marginal reg dual}
\begin{split}
    \sup  &\sum_{i=1}^m\int_{X^i}\phi^i(x^i)\dd\mu^i\\
    &-\eta \int_{X^1 \times...\times X^m} e^{\frac{\sum_{i=1}^m\phi^i(x^i) -c(x^1,...,x^m)}{\eta} }\dd(\otimes_{i=1}^m\mu^i)(x^1,\cdots,x^m).
\end{split}
\end{equation}
The regularized problem \eqref{pb:mmot-reg} and its dual \eqref{eqn: multi-marginal reg dual} arise frequently in computational work.  We note in particular that \eqref{pb:mmot-reg} is the minimization of a strictly convex functional and therefore admits a unique solution.  It is well known that as $\eta \rightarrow 0$, solutions of \eqref{pb:mmot-reg} and \eqref{eqn: multi-marginal reg dual} converge to solutions of \eqref{pb:mmot} and \eqref{eqn: multi-marginal dual}, respectively.
When each $X^i$ is a finite set (a case of particular interest in this paper), we obtain discrete versions of \eqref{pb:mmot} and its dual, which amount to linear programs, taking the forms

\begin{equation}
    \label{eq:discretePrimal}
    \inf\left\{\sum_{\bar x\in \times_{i=1}^m X^{i}}c(\bar x)\gamma_{\bar x}\;|\;\gamma\in\Pi(\mu^1,\cdots,\mu^m)  \right\},
\end{equation}
where $ \bar x=(x^1,\cdots,x^m) \in X^1 \times ....\times X^m$ and 
\begin{equation}
\label{pb:discreteDual}
\sup\{ \sum_{i=1}^m\sum_{x\in X^i}\phi^{i}_x\mu^i_x\;|\;(\phi^1,\cdots,\phi^m)\in\mathcal T\}
\end{equation}
where, if we identify functions $\phi^i:X^i \rightarrow \mathbb{R}$ with points in $\mathbb{R}^{|X^i|}$,  \[\mathcal T:=\{ \phi^i\in\R^{|X^i|}\;\forall i=1,\cdots,m,\sum_{i=1}^m\phi^i_{x^i}\leq c(x^1,\cdots, x^m),\; \forall (x^1,\cdots, x^m)\in \times_{i=1}^mX^i \}.\]
Notice that, if each $|X^i|=N$, in the case of the primal problem we have to deal with $N^m$ unknowns and $mN$ constraints, whereas in the dual problem there are $mN$ unknowns and $N^m$ constraints. In both cases we have to deal with the so called ``curse of dimensionality,'' namely the complexity of the problem increases exponentially with the number of marginals.\\ 

In this discrete setting, the entropy regularized problem \eqref{pb:mmot-reg} and its dual \eqref{eqn: multi-marginal reg dual} then become finite dimensional convex optimization problems:
\begin{equation}
    \label{pb:primalEntropic}
    \inf\left\{\sum_{\bar x\in \times_{i=1}^m X^{i}}c(\bar x)\gamma_{\bar x}+\eta [H(\gamma)-H(\otimes^m\mu^i)]\;|\;\gamma\in\Pi(\mu^1,\cdots,\mu^m), \right\}
\end{equation}
where $\eta>0$ and $H$ is the entropy with respect to uniform measure on the finite set $X$ (we suppress the subscript on $H$ indicating the reference measure in the finite case, as we will only deal with entropy relative to the uniform measure), $H(\gamma)=\sum_{\bar x\in \times_{i=1}^m X^{i}}h(\gamma_{\bar x})$, with
\[
h(t)=\begin{cases}
&t(\log(t)-1),\;t>0\\
&0,\;t=0\\
&+\infty,\; t<0,
\end{cases}
\]
and
\begin{equation}
    \label{pb:dualEntropic}
    \sup\left \{ \sum_{i=1}^m\sum_{x\in X^i}\phi^{i}_x\mu^i_x-\eta\sum_{\bar x\in \times_{i=1}^m X^{i}}\exp\Bigg(\frac {\sum_i\phi^i_{x^i}-c(\bar x)}{\eta}\Bigg)(\otimes^m\mu^i)_{\bar x} \right\}.
\end{equation}
We note in particular that \eqref{pb:dualEntropic} is an \emph{unconstrained} finite dimensional concave maximization problem.  Solutions may be computed using a multi-marginal version of the Sinkhorn algorithm \cite{benamouetalentropic,CuturiSinkhorn,peyre2017computational,Galichon-Entropic}, and  one can then recover the optimal $\gamma$ in \eqref{pb:primalEntropic} from the solutions $\phi^1,...,\phi^m$ to \eqref{pb:dualEntropic} via the well known formula:
$$
\gamma_{\bar x}=\exp{\bigg(\frac{\sum_{i=1}^m\phi^i_{x^i}  -c(\bar x)}{\eta}\bigg)}\mu^1_{x^1}\mu^2_{x^2}...\mu^m_{x^m}
$$
where $\bar x=(x^1,....,x^m)$.

\subsection{Pairwise costs}\label{sect: pairwise cost}
We are especially interested in this paper in  cost functions $c(x^1,....,x^m)$ involving pair-wise interactions, that is

\[
c(x^1,....,x^m) =\sum_{i< j}^m w(x^i,x^j).
\]
Such costs are ubiquitous in applications: for example, for systems of interacting classical particles in \cite{cotar2013density, buttazzo2012optimal}, $c$ is a pair-wise cost, with  $w(x-y)=\dfrac{1}{|x-y|}$, known as the Coulomb cost. 
 The case where $w(x,y) =|x-y|^2$ is the quadratic distance is well known to be equivalent to the Wasserstein barycenter problem (see Proposition 4.2 in \cite{Carlier_wasserstein_barycenter}), which has a wide variety of applications in statistics, machine learning and image processing, among other areas.\footnote{Pairwise costs with pair-dependent interactions, that is, costs of the form $c(x^1,....,x^m) =\sum_{i< j}^m w_{ij}(x^i,x^j)$ where the $w_{ij}$ are not necessarily the same for each choice of $i$ and $j$, also arise in the time discretization of Arnold's variational formulation of incompressible Euler equation \cite{Arnold66} and Brenier's relaxation of it \cite{brenier1989least}, and in inference problems for probabilistic graphical models \cite{Hassleretal21}.  The results in this section adapt immediately to such costs.}
 
Let us now consider 
costs $c_\epsilon$ of the form 
\begin{equation}\label{eqn: epsilon cost}
 c_\epsilon(x^1,\cdots,x^m):=\epsilon\sum_{i=2}^m\sum_{j=i+1}^mw(x^i,x^j) +\sum_{i=2}^mw(x^1,x^i).
 \end{equation}
It is clear that when $\epsilon=1$ we retrieve a pair-wise cost as defined above whereas in the limit $\epsilon\to 0$ we obtain a cost involving only the interactions between $x^1$ and the other $x^i$ individually.  
Later on, we will develop an ordinary differential equation that governs the evolution with $\epsilon$ of the solutions to the regularized dual problem \eqref{pb:dualEntropic}; the  results below assert that the initial condition for that equation (that is, the solutions when $\epsilon =0$) can be recovered by solving each of the individual two marginal problems between $\mu^1$ and $\mu^i$. 

In what follows, we will assume that each marginal $\mu^i$ is absolutely continuous with respect to a fixed based measure $\nu^i$ with density given by $\frac{\dd\mu^i}{\dd\nu^i}$.
\begin{proposition}\label{prop: limiting entropy} Assume that each marginal $\mu^i$ is absolutely continuous with respect to a fixed based measure $\nu^i$ with density given by $\frac{\dd\mu^i}{\dd\nu^i}(x^i)$: $\dd\mu^i(x^i) = \frac{\dd\mu^i}{\dd\nu^i}(x^i)\dd\nu^i(x^i)$.  
    Consider the regularized  problem \eqref{pb:mmot-reg} with limiting pairwise cost; that is, set $\epsilon =0$ in \eqref{eqn: epsilon cost} to obtain:
 \begin{equation}\label{eqn: limiting regularized problem}
     \min_{\gamma \in \Pi(\mu^1,\mu^2,...,\mu^m)}\int \sum_{i=2}^mw(x^1,x^i)\dd\gamma +\eta H_{ \otimes_{i=1}^m \mu^i}(\gamma).
 \end{equation}

    Let $\frac{\dd\bar \pi^i}{\dd(\nu^1 \otimes \nu^i)}$ be the density with respect to product measure $\nu^1(x^1) \otimes \nu^i(x^i)$ of the minimizer $\bar \pi^i=\frac{\dd\bar \pi^i}{\dd(\nu^1 \otimes \nu^i)}\nu^1\otimes\nu^i$   in the  regularized two marginal problem:
    \begin{equation}\label{eqn: limiting pairwise probelms}
\min_{\pi^i \in \Pi(\mu^1,\mu^i) }  \int w(x^1,x^i)\dd\pi^i(x^1,x^i) +\eta H_{\mu^1 \otimes \mu^i}(\pi^i). 
    \end{equation}
    Then the density $\frac{\dd\bar\gamma}{\dd(\otimes_{i=1}^m\nu^i)}$ of the optimal $\bar\gamma =\frac{\dd\bar\gamma}{\dd(\otimes_{i=1}^m\nu^i)}(\otimes_{i=1}^m\nu^i)$  in \eqref{eqn: limiting regularized problem} is given by 
    $$
    \frac{\dd\bar\gamma}{\dd(\otimes_{i=1}^m\nu^i)}(x^1,...,x^m)=\frac{\frac{\dd\bar\pi^2}{\dd(\nu^1 \otimes \nu^2)}(x^1,x^2)}{\frac{\dd\mu^1}{\dd\nu^1}(x^1)}\frac{\frac{\dd\bar\pi^3}{\dd(\nu^1 \otimes \nu^3)}(x^1,x^3)}{\frac{\dd\mu^1}{\dd\nu^1}(x^1)}...\frac{\frac{\dd\bar\pi^m}{\dd(\nu^1 \otimes \nu^m)}(x^1,x^m)}{\frac{\dd\mu^1}{\dd\nu^1}(x^1)}\frac{\dd\mu^1}{\dd\nu^1}(x^1).
    $$
\end{proposition}
\begin{proof}
    Choose any $\gamma =\frac{\dd\gamma}{\dd(\otimes_{i=1}^m \nu^i)}(\otimes_{i=1}^m \nu^i)\in \Pi(\mu^1,\mu^2,...,\mu^m)$ which is absolutely continuous with respect to $\otimes_{i=1}^m \nu^i$ and let $\pi^i(x^1,x^i) =\Big((x^1,..,x^m)\mapsto (x^1,x^i)\Big)_\#\gamma \in \Pi(\mu^1,\mu^i)$ be its twofold marginals.  Then
    \begin{eqnarray}
    H_{ \otimes_{i=1}^m \mu^i}(\gamma)&=& \int_{X^1\times ...\times X^m} [\log(\frac{\dd\gamma}{\dd(\otimes_{i=1}^m\nu^i)}(x^1,...,x^m)) -\sum_{i=1}^m \log(\frac{\dd\mu^i}{\dd\nu^i}(x^i))]\dd\gamma(x^1,...,x^m)\nonumber\\
    &=&\int_{X^1\times ...\times X^m} [\log(\frac{\dd\gamma}{\dd(\otimes_{i=1}^m\nu^i)}(x^1,...,x^m))] \dd\gamma(x^1,....x_N) -\sum_{i=1}^m H_{\nu^i}(\mu^i),\label{eqn: relative entropy decomposition}
    \end{eqnarray}
    where each $H_{\nu^i}(\mu^i) =\int_{X^i}\log(\frac{\dd\mu^i}{\dd\nu^i}(x^i))\dd\mu^i(x^i)$ is constant throughout $\Pi(\mu^1,...,\mu^m)$.
    
    Now disintegrating $\gamma =\gamma_{x^1}(x^2,...,x^m)\otimes \mu^1(x^1)$ with respect to its first marginal $\mu^1$, we note that $\gamma_{x^1}(x^2,...,x^m) = \frac{\dd\gamma}{\dd(\otimes_{i=1}^m\nu^1)}(x^1,x^2,...,x^m)\frac{1}{\frac{\dd\mu^1}{\dd\nu^1}(x^1)}(\otimes_{i=2}^m\nu^i) $ and so
    \begin{equation}
     \label{eqn: conditional entropy}
    \begin{split}
    &\int_{X^1\times ...\times X^m} \log\bigg(\frac{\dd\gamma}{\dd(\otimes_{i=1}^m\nu^i)}\bigg) \dd\gamma \\ 
     =&\int_{X^1\times ...\times X^m} \log\bigg(\frac{\dd\gamma_{x^1}}{\dd(\otimes_{i=2}^m\nu^i)}\bigg) \dd\gamma +H_{\nu^1}(\mu^1)\\
    =&\int_{X^1}\int_{X^2\times ...\times X^m} \log\bigg(\frac{\dd\gamma_{x^1}}{\dd(\otimes_{i=2}^m\nu^i)}\bigg) \dd\gamma_{x^1}\dd\mu^1 +H_{\nu^1}(\mu^1)\\
    =&\int_{X^1}H_{\otimes_{i=2}^m \nu_{i}}(\gamma_{x^1})\dd\mu^1 +H_{\nu^1}(\mu^1)
    \end{split}
  \end{equation}
    where $H_{\otimes_{i=2}^m \nu_{i}}(\gamma_{x^1})$ is the entropy of $\gamma_{x^1}$ with respect to $\otimes_{i=2}^m \nu_{i}$ and $H_{\nu^1}(\mu^1)$ is the entropy of $\mu^1$ with respect to $\nu^1$.  Now note that if we disintegrate each $\pi^i =\pi^i_{x^1}(x^i)\otimes\mu^1(x^1)$ with respect to $\mu^1$, then for each fixed $x^1$, the conditional probability $\pi^i_{x^1}$ is the $i$th marginal of $\gamma_{x^1}$ and so
  
   \begin{eqnarray}
   \int_{X^1}H_{\otimes_{i=2}^m \nu_{i}}(\gamma_{x^1})\dd\mu^1(x^1) &\geq& \int_{X^1} \sum_{i=2}^m H_{\nu^i}(\pi^i_{x^1}) \dd\mu^1(x^1)\nonumber\\
   &=&\sum_{i=2}^m\int_{X^1}  H_{\nu^i}(\pi^i_{x^1}) \dd\mu^1(x^1)\nonumber\\
   &=&\sum_{i=2}^m [H_{\nu^1 \otimes \nu^i}(\pi^i) -H_{\nu^1}(\mu^1)]\nonumber\\
  &=&\sum_{i=2}^m [H_{\mu^1 \otimes \mu^i}(\pi^i) +H_{\nu^i}(\mu^i)], \label{eqn: conditional entorpy inequality}
   \end{eqnarray}
    where the equality $H_{\nu^1 \otimes \nu^i}(\pi^i)=\int_{X^1}  H_{\nu^i}(\pi^i_{x^1}) \dd\mu^1(x^1) +H_{\nu^1}(\mu^1)$ in the second to last line follows very similarly to the derivation of \eqref{eqn: conditional entropy} above and the equality  $H_{\nu^1 \otimes \nu^i}(\pi^i) -H_{\nu^1}(\mu^1)=H_{\mu^1 \otimes \mu^i}(\pi^i) +H_{\nu^i}(\mu^i)$ in the last line follows very similarly to 
the derivation of \eqref{eqn: relative entropy decomposition}.
    Therefore, combining \eqref{eqn: relative entropy decomposition}, \eqref{eqn: conditional entropy} and \eqref{eqn: conditional entorpy inequality}, we get
  \begin{eqnarray*}
     & \int_{X^1 \times ...X^m} \sum_{i=1}^mw(x^1,x^i)\dd\gamma +\eta H_{\otimes_{i=1}^m \mu_{i}}(\gamma)
     =  \int_{X^1 \times X^i} \sum_{i=1}^mw(x^1,x^i)\dd\pi^i +\eta H_{\otimes_{i=1}^m \mu_{i}}(\gamma) \\
      &\geq \int_{X^1 \times X^i} \sum_{i=2}^mw(x^1,x^i)\dd\pi^i+\sum_{i=2}^m [H_{\mu^1 \otimes \mu^i}(\pi^i) +H_{\nu^i}(\mu^i)]+H_{\nu^1}(\mu^1) -\sum_{i=1}^mH_{\nu^i}(\mu^i)\\
      &\geq  \int_{X^1 \times X^i} \sum_{i=2}^mw(x^1,x^i)\dd\bar \pi^i+\sum_{i=2}^m H_{\mu^1 \otimes \mu^i}(\bar \pi^i)\
    \end{eqnarray*}
    by optimality of $\bar \pi$.
    We have equality in the last line if and only if $\pi^i =\bar \pi^i$ for each $i$, and equality in the line above if and only if $\mu^1$ almost every $\gamma_{x^1}$ couples the $\pi^i_{x^i}$ independently; this yields the desired result.
\end{proof}
Note in particular that this result allows us to recover the  solution to problem \eqref{pb:primalEntropic} with cost \eqref{eqn: epsilon cost}, when $\epsilon =0$,   by solving $m-1$ individual regularized two marginal optimal transport problems. In the following section, we will develop a dynamical approach to solve the dual problem \eqref{pb:dualEntropic} to \eqref{pb:primalEntropic} for cost \eqref{eqn: epsilon cost} with $\epsilon >0$.  Our initial condition will be the dual potentials when $\epsilon =0$, which we can obtain from the corresponding two marginal dual potentials, as the following corollary confirms. 
\begin{corollary}\label{cor: epsilon =0 dual potentials}
    Assume each $\mu^i$ is absolutely continuous with respect to a given reference measure $\nu^i$.  For each $i=2,...m$, let $\psi^i(x^1), \phi^i(x^i),$ solve the regularized two marginal dual problem \eqref{pb:dualEntropic} between marginals $\mu^1$ and $\mu^i$ with cost function $w(x^1,x^i)$.  Then $\phi^1(x^1), \phi^2(x^2),...,\phi^m(x^m)$, with $\phi^1(x^1) =\sum_{i=2}^m\psi^i(x^1)$ solve the regularized dual \eqref{pb:dualEntropic} with marginals $\mu^1,\mu^2,....,\mu^m$ and cost $c(x^1,...,x^m)=\sum_{i=2}^mw(x^1,x^i)$.
\end{corollary}

\begin{proof}
    We have that for each $i$, the optimizer $\bar \pi^i$ in the regularized two marginal primal problem satisfies
    $$
   \frac{\dd\bar\pi^i}{\dd(\nu^1 \otimes \nu^i)}(x^1,x^i) = e^{\frac{\phi^i(x^i) +\psi^i(x^i) -w(x^1,x^i)}{\eta}}\frac{\dd \mu^1}{\dd\nu^1}\frac{\dd \mu^i}{\dd\nu^i}.
    $$
    By the preceding proposition, the optimizer in the regularized multi-marginal problem \eqref{eqn: limiting regularized problem} satisfies
  \begin{eqnarray*}
   \frac{\dd\bar\gamma}{\dd(\otimes_{i=1}^m\nu^i)}(x^1,x^2,...,x^m)&=&\frac{\frac{\dd\bar\pi^2}{\dd(\nu^1 \otimes \nu^2)}(x^1,x^2)}{\frac{\dd\mu^1}{\dd\nu^1}(x^1)}\frac{\frac{\dd\bar\pi^3}{\dd(\nu^1 \otimes \nu^3)}(x^1,x^3)}{\frac{\dd\mu^1}{\dd\nu^1}(x^1)}...\frac{\frac{\dd\bar\pi^m}{\dd(\nu^1 \otimes \nu^m)}(x^1,x^m)}{\frac{\dd\mu^1}{\dd\nu^1}(x^1)}\frac{\dd\mu^1}{\dd\nu^1}(x^1)\\
    &=& e^{\frac{\sum_{i=2}^m[\phi^i(x^i) +\phi^i(x^1) -w(x^1,x^i)]}{\eta}}\frac{\dd \mu^2}{\dd\nu^2}\frac{\dd \mu_3}{\dd\nu^3}...\frac{\dd \mu^m}{\dd\nu^m}\frac{\dd \mu^1}{\dd\nu^1}\\
    &=&e^{\frac{\sum_{i=1}^m\phi^i(x^i) -\sum_{i=2}^mw(x^1,x^i)}{\eta}}\frac{\dd \mu^1}{\dd\nu^1}\frac{\dd \mu^2}{\dd\nu^2}\frac{\dd \mu_3}{\dd\nu^3}...\frac{\dd \mu^m}{\dd\nu^m}
  \end{eqnarray*}
  This is exactly the first order condition identifying the regularized potentials for the multi-marginal regularized problem with cost $\sum_{i=2}^mw(x^1,x^i)$.
\end{proof}
 \begin{remark}We note that the cost \eqref{eqn: epsilon cost} at $\epsilon =0$ falls under the class of tree-structured costs investigated in \cite{Hassleretal21}. These problems are known to correspond to probabilistic graphical models; in our setting, the corresponding graph is in fact a star graph centered at $x_1$.  From an algorithmic perspective, such costs are desirable since one can solve the multi-marginal problem by solving a two marginal problem for each edge of the graph (which is $m-1$ in our case); as was shown in \cite{Hassleretal21}, considering the regularized multi-marginal optimal transport problem in fact has some computational advantages over the corresponding series of pairwise regularization.  This suggests a connection between that line of research and Proposition \ref{prop: limiting entropy} and Corollary \ref{cor: epsilon =0 dual potentials}, since these results essentially assert an equivalence between the regularized multi-marginal problem \eqref{eqn: limiting regularized problem} and the $m$ pairwise problems \eqref{eqn: limiting pairwise probelms}; it is not clear to us whether the techniques in \cite{Hassleretal21} can be used to provide an alternative proof of these results.\end{remark}

\section{An ODE characterisation of discrete  multi-marginal optimal transport}\label{sect: ode approach}
We now turn our attention to developing an ODE for the Kantorovich potentials after discretizing the marginals.  Working with the regularized discrete problem \eqref{pb:primalEntropic} and its dual \eqref{pb:dualEntropic} with pairwise cost \eqref{eqn: epsilon cost}, we make the following, standing assumptions throughout this section:

\begin{enumerate}
\item (Equal marginals) All the marginals are equal $\mu^i=\rho =\sum_{x \in X}\rho_x\delta_x$, where $X$ is a finite subset of $\mathbb{R}^d$,
\item (Symmetric cost) The two body cost $w$ is symmetric $w(x,y) =w(x,y)$. 
    \item (Finite cost) The two body cost function $w:X \times X \rightarrow \mathbb{R}$ is everywhere real-valued.
\end{enumerate}
 A motivating example of a pairwise, symmetric two body cost arises in Density Functional Theory where the  cost is given by $w(x,y)=\dfrac{1}{|x-y|}$; in problems with this cost, the marginals are typically also identical.  
 The cost does not satisfy the finiteness hypothesis, but one can consider a truncation $w(x,y)=\min\bigg(\dfrac{1}{|x-y|},C\bigg)$ cost; it is known that the solution stays away from the diagonal, and for sufficiently large $C$, the solution with the truncated cost coincides with the solution for the original Coulomb cost (for instance see \cite{buttazzo2018continuity}).

\begin{remark}\label{rem: dropping symmetry}One could dispose of the equal marginal and symmetric cost assumptions.  Analogues of the results proved below would sill hold; one could characterize the solution to the regularized dual problem \eqref{pb:dualEntropic} by an ODE, and prove that this ODE is well-posed.   In this setting, one could also work with pair-dependent costs of the form $c(x^1,....,x^m) =\sum_{i< j}^m w_{ij}(x^i,x^j)$, as discussed briefly in Section \ref{sect: pairwise cost}.  As we will see below, however, solving the problem numerically becomes more feasible  under the hypotheses above, as the solution can be characterized by a single Kantorovich potential, and so the resulting ODE is an equation on $\mathbb{R}^N$, where $N$ is the number of points in the support of the marginal.   With unequal marginals and a non-symmetric cost, one would require $m-1$ independent Kantorovich potentials to fully characterize the solution; if each marginal is supported on $N$ points, this would lead to an $(m-1)N$ dimensional system of ODEs.  For the sake of simplicity we do not re-write the proof of the well-posedness since it works exactly as in the symmetric framework; however,  we refer the reader to section \ref{sec:euler} where we present numerical simulations for the (non-symmetric) case of the multi-marginal problem associated to the relaxed formulation of incompressible Euler equation introduced by Brenier. 
\end{remark}
 \subsection{Formulation of the ODE problem}
Notice now that although the cost \eqref{eqn: epsilon cost} at $\epsilon=1$ is symmetric in the variables $x^1,x^2,...,x^m$, the one at $\epsilon<1$ is not. It is, however, symmetric in the variables $x^2,...,x^m$; this means that the optimal $\phi^i$ in \eqref{pb:dualEntropic} satisfy $\phi^i=\phi^j=\phi$ for $i,j \geq 2$ and so, setting $\phi^1 =\psi,$ we can rewrite \eqref{pb:dualEntropic} as 
\begin{equation}
\label{pb:discreteDualRe}
\inf_{\phi,\psi: X \rightarrow \mathbb{R}}\left\{\Phi(\phi,\psi,\epsilon)\right\},
\end{equation}
where 
\[ \Phi(\phi,\psi,\epsilon):=-(m-1)\sum_{x\in X}\phi_x\rho_x-\sum_{x\in X}\psi_x\rho_x+\eta\sum_{\bar x\in  X^{m}}e^{\Bigg(\frac {\sum_{i=2}^m\phi_{x^i}+\psi_{x^1}-c_\epsilon(\bar x)}{\eta}\Bigg)}\otimes^m\rho.\]

\begin{remark}[Notation]
Recall that we use the notation $\bar x$ to represent a point in a product space, such as $\bar x =(x^1,\cdots,x^{m})\in X^m$, as above, or, as will often be the case in what follows, $\bar x =(x^1,\cdots,x^{m-1}) \in X^{m-1}$.  We introduce the following notation to represent corresponding products of the densities:
\[ \tilde \rho_{\bar x}=(\otimes^{m-1}\rho)_{\bar x}=\otimes_{i=1}^{m-1}\rho_{x^i} \]
\end{remark}

Since the functional $\Phi(\phi,\psi,\epsilon)$ is convex on the set $\{\phi,\psi: X \rightarrow \mathbb{R}\} \approx \mathbb{R}^{2|X|}$, as the sum of a linear and an exponential function,  optimal solutions $(\phi^*,\psi^*)$ can be characterized by the first order optimality conditions $\nabla_\phi\Phi=\nabla_\psi\Phi=0$, or (component-wise):
\[\phi^*_z=-\eta\log\Bigg(\sum_{\bar x\in X^{m-1}}\exp\Bigg(\frac{\sum_{i=2}^{m-1}\phi^*_{x^i}+\psi^*_{x^1}-c_\epsilon(\bar x,z)}{\eta}\Bigg)\tilde \rho_{\bar x}\Bigg)  \]
and
\begin{equation}\label{eqn: psi from phi}
\psi^*_z=-\eta\log\Bigg(\sum_{\bar x\in X^{m-1}}\exp\Bigg(\frac{\sum_{i=2}^{m}\phi^*_{x^i}-c_\epsilon(z,\bar x)}{\eta}\Bigg)\tilde \rho_{\bar x}\Bigg). 
\end{equation}
In particular, note that \eqref{eqn: psi from phi} allows us to express the optimal $\psi^*$ in \eqref{pb:discreteDualRe} in terms of the optimal $\phi^*$, after which \eqref{pb:discreteDualRe} reduces to the following optimization problem
\begin{equation}
\label{pb:discreteDualRered}
\inf_{\phi: X \rightarrow \mathbb{R}}\left\{\tilde\Phi(\phi,\epsilon)\right\},
\end{equation}
where
\[\tilde\Phi(\phi,\epsilon):=-(m-1)\sum_{x\in X}\phi_x\rho_x+\eta\sum_z\log\Bigg(\sum_{\bar x\in X^{m-1}}e^{\Bigg(\frac{\sum_{i=2}^{m}\phi_{x^i}-c_\epsilon(z,\bar x)}{\eta}\Bigg)}\tilde \rho_{\bar x}\Bigg)\rho_z . \]

\begin{remark}[LogSumExp and convexity]
The function \[\phi\mapsto\log\Bigg(\sum_{\bar x\in X^{m-1}}\exp\Bigg(\frac{\sum_{i=2}^{m}\phi_{x^i}-c_\epsilon(z,\bar x)}{\eta}\Bigg)\tilde \rho_{\bar x}\Bigg):=LSE_{c_\epsilon}(\phi)_z\]
is also known as Log-Sum-Exp function (LSE). By using the Hölder inequality one can easily show that the Log-Sum-Exp is convex.
\end{remark}
It is well known that the solution to \eqref{pb:dualEntropic} is unique up to the addition of constants $\phi^i \mapsto \phi^i+C^i$ adding to $0$, $\sum_{i=1}^mC^i=0$; thus, solutions to \eqref{pb:discreteDualRered} are unique up to the addition of a single constant, $\phi \mapsto \phi+C$.  
We therefore impose the normalization 
\begin{equation}\label{eqn: normalization}
    \phi_{x_0}=0 
\end{equation} 
for all $\epsilon\in[0,1]$ and a fixed $x_0\in X$. 

The  problem \eqref{pb:discreteDualRered}, restricted to $\phi$'s satisfying \eqref{eqn: normalization} then has a unique solution; the function $\tilde\Phi(\cdot,\epsilon)$ is strictly convex when restricted to this set, and the solution $\phi^*=\phi(\epsilon)$ 
 can be characterized by the optimality condition $\nabla_\phi\tilde\Phi(\phi^*,\epsilon)=0$, where each component of the gradient is given by
\begin{equation}
\label{eq:grad}
    \dfrac{\partial}{\partial\phi_z}\tilde\Phi=-(m-1)\rho_z+(m-1)e^{\phi_z/\eta}\rho_z\sum_y\sum_{\bar x\in X^{m-2}}e^{\Bigg(\frac{\sum_{i=3}^{m}\phi_{x^i}-c_\epsilon(y,z,\bar x)}{\eta}\Bigg)}(\otimes^{m-2}\rho_{\bar x})\bar \rho_y. 
\end{equation}

where
\[\bar \rho_y=\dfrac{\rho_y}{\sum_{\bar x\in X^{m-1}}\exp\Bigg(\frac{\sum_{i=2}^{m}\phi_{x^i}-c_\epsilon(y,\bar x)}{\eta}\Bigg)\tilde\rho_{\bar x}}.\]
Our numerical method consists then in solving an ODE for the evolution of $\phi(\epsilon)$ obtained by differentiating 
\begin{equation}
    \label{eqn: first order optimality}
\nabla_{\phi}\tilde\Phi(\phi(\epsilon),\epsilon)=0
\end{equation}
with respect to $\epsilon$:
\begin{equation}
\label{eq:ODE}
    \frac{\partial}{\partial \epsilon}\big[\nabla_{\phi}\tilde\Phi(\phi,\epsilon)\big]|_{\phi=\phi(\epsilon)}+D^2_{\phi,\phi}\tilde\Phi(\phi(\epsilon),\epsilon)\frac{\dd\phi}{\dd\epsilon}(\epsilon)=0.
\end{equation}
 If the  pure second derivatives  with respect to $\phi$ as well as the mixed second derivatives with respect to $\phi$ and $\epsilon$ exist and are Lipschitz, and  the Hessian with respect to $\phi$ is invertible, we will obtain a characterization of $\phi$ as the solution to the following well-posed Cauchy problem:
 \begin{equation}
     \label{eq:PdCprelim}
     \begin{cases}
      &\dfrac{\dd\phi}{\dd\epsilon}(\epsilon)=-[D^2_{\phi,\phi}\tilde\Phi(\phi(\epsilon),\epsilon)]^{-1}\dfrac{\partial}{\partial \epsilon}\nabla_{\phi}\tilde\Phi(\phi(\epsilon),\epsilon),\\
      &\phi(0)=\phi_w,
      \end{cases}
 \end{equation}
where, by Corollary \ref{cor: epsilon =0 dual potentials}, the initial value $\phi(0)$ of $\phi$ when $\epsilon =0$ coincides with  $\phi_w$, the optimal potential for the two marginal optimal transport problem with cost $w$.

The next section is devoted to proving these properties.

\subsection{Well posedness of the ODE}
We refer the reader to appendix \ref{app:seconderv} for the computation of the second pure and mixed derivatives with respect to $\phi$ and the second mixed derivative with respect to $\phi$ and  $\epsilon$.
In order to prove invertibility of $D^2_{\phi,\phi}\tilde \Phi$ and well posedness of the ODE we need some lemmas giving uniform bounds on the potential $\phi$ and the eigenvalues of $D^2_{\phi,\phi}\tilde \Phi$. We highlight that the following arguments are similar to (and largely inspired by) those in \cite{carlier2021linear} (the main differences lie in the fact that we  re-write the dual problem by using the Log-Sum-Exp function).
\begin{lemma}
\label{lemma:bounds}
    Let $c_\epsilon$ satisfy the boundedness assumption $\norm{c_\epsilon}_\infty\leq M$, $\forall \epsilon\in[0,1]$ \footnote{Note that the boundedness $\norm{c_\epsilon}_\infty\leq M$ for some $M>0$ follows immediately from our finite cost assumption on the finite set $X^m$. }.  Then the maximizer $\phi(\epsilon)$ of \eqref{pb:discreteDualRered} subject to the normalization constraint \eqref{eqn: normalization} satisfies
    \[ \norm{\phi(\epsilon)}_\infty\leq 4M. \]
\end{lemma}
\begin{proof}
    By the first order optimality condition $\nabla_\phi\tilde\Phi =0$ for \eqref{pb:discreteDualRered} we deduce that each component of $\phi(\epsilon)$ is given by
    \[ \phi_z=-\eta\log\Bigg(\sum_y\sum_{\bar x\in X^{m-2}}\exp\Bigg(\frac{\sum_{i=3}^{m}\phi_{x^i}-c_\epsilon(y,z,\bar x)}{\eta}\Bigg)(\otimes^{m-2}\rho)_{\bar x}\bar \rho_y\Bigg) \]
    It is is easy to see that $\bar \rho_y$ can be bounded as follows
    \[ \dfrac{e^{-M/\eta}\rho_y}{\sum_{\bar x\in X^{m-1}}\exp\Bigg(\frac{\sum_{i=2}^{m}\phi_{x^i}}{\eta}\Bigg)\tilde\rho_{\bar x}}\leq \bar \rho_y\leq \dfrac{e^{M/\eta}\rho_y}{\sum_{\bar x\in X^{m-1}}\exp\Bigg(\frac{\sum_{i=2}^{m}\phi_{x^i}}{\eta}\Bigg)\tilde\rho_{\bar x}}. \]
Since we have imposed the normalization $\phi_{x_0}=0$ we get
    \[
    \begin{split}
    \phi_z=\phi_z-\phi_{x_0}&\leq-\eta\log\Bigg(e^{-2M/\eta}\frac{\sum_{\bar x\in X^{m-2}}\exp\Bigg(\frac{\sum_{i=3}^{m}\phi_{x^i}}{\eta}\Bigg)(\otimes^{m-2}\rho)_{\bar x}}{\sum_{\bar x\in X^{m-1}}\exp\Bigg(\frac{\sum_{i=2}^{m}\phi_{x^i}}{\eta}\Bigg)\tilde\rho_{\bar x}}  \Bigg)\\
    &+\eta\log\Bigg(e^{2M/\eta}\frac{\sum_{\bar x\in X^{m-2}}\exp\Bigg(\frac{\sum_{i=3}^{m}\phi_{x^i}}{\eta}\Bigg)(\otimes^{m-2}\rho)_{\bar x}}{\sum_{\bar x\in X^{m-1}}\exp\Bigg(\frac{\sum_{i=2}^{m}\phi_{x^i}}{\eta}\Bigg)\tilde\rho_{\bar x}}  \Bigg),\\
    &\leq 4M,
    \end{split}\]
    and the desired result immediately follows.
\end{proof}
Having established the above bounds, we aim to prove the well posedness of the Cauchy problem      \eqref{eq:PdCprelim} on the set 
\begin{equation}\label{eqn:def of U}
U:=\{\phi: X \rightarrow \mathbb{R}\;|\;\phi_{x_0}=0,\;   \norm{\phi}_\infty\leq 4M\}.
\end{equation}

\begin{lemma}\label{lem: lipschitz second derivatives}
$D^2_{\phi,\phi}\tilde\Phi(\phi,\epsilon)$ and $\frac{\partial}{\partial \epsilon}\nabla_{\phi}\tilde\Phi(\phi,\epsilon)$ are Lipschitz with respect to $\phi$ on $U$.
\end{lemma}
\begin{proof}
    This immediately follows from the fact that the the second pure and mixed derivatives computed in Appendix \ref{app:seconderv} are easily seen to be $C^1$, and their derivatives are all bounded on $U$.
\end{proof}
In order to prove the invertibility of $D^2_{\phi,\phi}\tilde\Phi$ we need the following lemma assuring the strong convexity of the Log-Sum-Exp function on the set $U$.
\begin{lemma}
\label{lem:strong_convexity}
Let $\Psi: \tilde U \rightarrow \mathbb{R}$ be defined on

$$
\tilde U_C=\{\theta:X^{m-1} \rightarrow \mathbb{R} \;|\; \theta_{\bar x_0} =0,\;\norm{\theta}_{\infty}<C\}.
$$
where $\bar x_0=(x_0,...,x_0) \in X^{m-1}$, by $\Psi(\theta)=\sum_{y\in X}\log\Big(\sum_{\bar x\in X^{m-1}}e^{\theta_{\bar x}-c_\epsilon(y,\bar x)}\tilde \rho_{\bar x} \Big)\rho_y$.

Then $\Psi$ is $\beta$-strongly convex for some $\beta >0$.
\end{lemma}

\begin{proof}

It is enough to show strong convexity on this set of the function \[f_y:\theta\in \tilde U_C\mapsto\log\Big(\sum_{\bar x\in X^{m-1}}e^{\theta_{\bar x}-c_\epsilon(y,\bar x)}\tilde\rho_{\bar x} \Big) =\log\Big(e^{-c_\epsilon(y, \bar x_0)}\tilde \rho_{\bar x_0}+\sum_{\bar x\in X^{m-1}\setminus \{\bar x_0\}}e^{\theta_{\bar x}-c_\epsilon(y,\bar x)}\tilde\rho_{\bar x} \Big)\] for a fixed $y$.

Enumerating the set $ X^{m-1}\setminus \{\bar x_0\}$ of independent variables as $\bar x_j$ for $j \in (1,...,K)$ with $K= |X|^{m-1}-1$, and denoting $z^j =e^{\theta_{\bar x}-c_\epsilon(y,\bar x)}\tilde\rho_{\bar x}$ the Hessian of $f_y$ is
$$
\frac{1}{\Big(e^{-c_\epsilon(y, \bar x_0)}\tilde \rho_{\bar x_0}+\sum_{j}z^j\Big)^2}\Big(-z\otimes z +\diag(z)(\sum_{j}z^j +e^{-c_\epsilon(y, \bar x_0)}\tilde \rho_{\bar x_0})\Big)
$$
The first two terms together constitute a positive semi-definite matrix (see \cite{BoydVandenberghe04}, p.74), while the third is positive definite, with lower bound 

\[\beta =\frac{e^{-C-2M}\tilde \rho_{\bar x_0}\min_{\bar x} \tilde \rho_{\bar x}}{(e^{C+M}\tilde \rho_{\bar x_0} +\sum_{\bar x\in X^{m-1}\setminus \{\bar x_0\}}e^{C+M}\tilde\rho_{\bar x})^2}=e^{-4M-3C}\tilde \rho_{\bar x_0}\min_{\bar x} \tilde \rho_{\bar x}.\]

It follows that $f_y$, and therefore $\Psi$, is $\beta$-convex on $\tilde U$.

\end{proof}

 \begin{lemma}\label{lem: strong convexity}

Let $\Lambda: U \rightarrow \tilde U_{4(m-1)M}$ be the linear mapping defined by $\Lambda(\phi)_{\bar x}=\phi_{x^1}+\cdots+\phi_{x^{m-1}}$ $\forall \bar x\in X^{m-1}$. Then $\tilde\Psi(\phi):=\Psi(\Lambda(\phi))$ is $\alpha-$strongly convex on $U$.
\end{lemma}
 \begin{proof}
By the linearity of $\Lambda$, one gets that, for  $\phi\in U$, $D^2\tilde\Psi(\phi)(v,v)=D^2\Psi(\Lambda(\phi))(\Lambda(v),\Lambda(v))$ for all $v\in U$. Thus,
\[ D^2\tilde\Psi(\phi)(v,v)=D^2\Psi(\Lambda(\phi))(\Lambda(v),\Lambda(v))\geq  \beta \norm{\Lambda(v)}^2. \]
Since $\norm{\Lambda(v)}^2\geq \sum_{x\in X}\norm{(m-1)v_x}^2$ we finally get
\[ D^2\tilde\Psi(\phi)(v,v)\geq \alpha \norm{v}^2,\]
with $\alpha=\beta (m-1)^2>0$, proving the $\alpha-$strong convexity of $\tilde\Psi$.
\end{proof}
\begin{remark}
The $\alpha$ obtained in the Lemma above is not optimal; indeed we would have obtained a better lower bound on the eigenvalues of $D^2_{\phi,\phi}\tilde\Psi$ by computing the smallest eigenvalue of $\Lambda^*\Lambda$. Moreover, in the previous lemma we take, for simplicity, $\eta=1$, otherwise the parameter $\alpha$ would have taken the form $\alpha=e^{(-4C-3M)/\eta}(m-1)^2$. Notice now that $\alpha$ approaches to 0 as $\eta\to 0$, meaning the the condition number of the Hessian of $\tilde\Phi$  explodes. Namely, this will produce numerical instabilities. 
\end{remark}
It easily follows from the previous lemma that $D^2_{\phi,\phi}\tilde\Phi =D^2_{\phi,\phi}\tilde\Psi$ is invertible on the set $U$; we can then state the following result on the well posedness of \eqref{eq:PdCprelim}.
\begin{theorem}
Let $\phi(\epsilon)$ be the solution to \eqref{pb:discreteDualRered} for all $\epsilon\in[0,1]$. Then $\epsilon\mapsto\phi(\epsilon)$ is $\mathcal C^1$ and is the unique solution to the Cauchy problem
 \begin{equation}
     \label{eq:PdC}
     \begin{cases}
      &\dfrac{\dd\phi}{\dd\epsilon}(\epsilon)=-[D^2_{\phi,\phi}\tilde\Phi(\phi(\epsilon),\epsilon)]^{-1}\dfrac{\partial}{\partial \epsilon}\nabla_{\phi}\tilde\Phi(\phi(\epsilon),\epsilon),\\
      &\phi(0)=\phi_w,
      \end{cases}
 \end{equation}
 where $\phi_w$ is the optimal solution to \eqref{pb:dualEntropic} with cost $w$ and two marginals equal to $\rho$.
\end{theorem}

\begin{proof}
As $\phi(\epsilon)$ minimizes $\tilde \Phi(\cdot,\epsilon )$ for each fixed $\epsilon$, we clearly have \eqref{eqn: first order optimality}.  Since $\tilde \Phi$ is clearly twice differentiable with respect to $\phi$ and $\epsilon$ and $D^2_{\phi\phi}\tilde \Phi$ is invertible by Lemma \ref{lem: strong convexity}, the Implicit Function Theorem then implies that $\epsilon \mapsto \phi(\epsilon)$ is $C^1$ and satisfies \eqref{eq:ODE}, or equivalently, \eqref{eq:PdC}.  

Since $D^2_{\phi,\phi}\tilde \Phi$ and $\frac{\partial }{\partial \epsilon}\nabla _{\phi}\tilde \Phi$ are Lipschitz continuous with respect to $\phi$ on $U$ by Lemma \ref{lem: lipschitz second derivatives} and clearly continuous with respect to $\epsilon$, and since $D^2_{\phi,\phi}\tilde \Phi$ is uniformly positive definite by Lemma \ref{lem: strong convexity}, we have that 
$$
(\phi, \epsilon) \mapsto -[D^2_{\phi,\phi}\tilde\Phi(\phi(\epsilon),\epsilon)]^{-1}\dfrac{\partial}{\partial \epsilon}\nabla_{\phi}\tilde\Phi(\phi(\epsilon),\epsilon)
$$
is Lipschitz with respect to $\phi$ and continuous with respect to $\epsilon$ on $U$. Since by Lemma \ref{lemma:bounds} $\phi(\epsilon) \in U$ for all $\epsilon,$ the Cauchy-Lipschitz Theorem then implies uniqueness of the solution to \eqref{eq:PdC} on $U \times [0,1]$, as desired.

\end{proof}

\section{Algorithm and simulation}\label{algorithm}
In this subsection we present some 
numerical simulations\footnote{The simulations have been performed in Python on 2 GHz Quad-Core Intel Core i5 MacBook Pro.} obtained by 
discretizing the above ODE.

 The algorithm consists in discretizing \eqref{eq:PdC} by an explicit Euler scheme (notice that one could also use some high order method for the ODEs). Let $h$ be the step size and set $\phi(0)=\phi_w$ the solution of a 2 marginal problem with cost $w$, then the $\phi$ can be defined inductively as detailed in \ref{algo}.
\begin{algorithm}
    \caption{Algorithm to compute the $\phi$ via explicit Euler method}\label{algo}
    \begin{algorithmic}[1]
     \Require $\phi(0)=\phi_w$
     \While{$||\phi^{(k+1)}-\phi^{(k)}||<$tol}
     \State $D^{(k)}:=D^{2}_{\phi,\phi}\tilde\Phi(\phi^{(k)},kh)$
     \State $b^{(k)}:=-\dfrac{\partial}{\partial \epsilon}\nabla_{\phi}\tilde\Phi(\phi^{(k)},kh)$
     \State Solve $D^{(k)}z=b^{(k)}$
     \State $\phi^{(k+1)} = \phi^{(k)}+hz$
     \EndWhile
    \end{algorithmic}\label{alg: pseudo-code}
\end{algorithm}
 


Notice that by the regularity we have proved above, we can conclude that the Euler scheme converges linearly. Moreover, the uniform error between the discretized solution obtained via the scheme and the solution to the ODE is $O(h)$.

\begin{remark}[Complexity in the marginals]
We highlight that this ODE approach  does not overpass the  NP hardness (see \cite{altschuler2021hardness}) of some multi-marginal problems, as the one with the Coulomb cost,  and it still suffers the exponential complexity in the number of  marginals. However,  since the ODE is smooth, one can try apply an high order method  (we will do a careful analysis of it later in the section) which can converge quickly in the number of iterations and have a computational time competitive with respect to Sinkhorn.  

We also note that, if the symmetry assumptions are dropped, as discussed in Remark \ref{rem: dropping symmetry}, one needs to solve a system of $(N-1)m$ ODEs.  The matrix $D^{(k)}$ in Algorithm \ref{alg: pseudo-code} is then be $(N-1)m \times (N-1)m$ instead of $m \times m$, increasing the leading term in the complexity by the relatively manageable multiplicative factor $(N-1)^2$.  Numerical results for one such case are presented in Section \ref{sec:euler} below.
\end{remark}

 In Figure  \ref{fig1a} we plot the convergence order for the Euler scheme described above. The error is computed with respect to the solution to \eqref{pb:discreteDualRered} computed via a gradient descent method with backtracking. Notice that the regularity of the objective function and the boundedness of the Hessian guarantee the convergence of the method. For these simulations we have taken $m=3$, the uniform measure on $[0,1]$ uniformily discretized with $100$ gridpoints and the pairwise interaction $w(x,y)=-\log(0.1+|x-y|)$. Moreover, since the RHS in \eqref{eq:PdC} is regular one can try to apply an high order scheme to solve the Cauchy problem. In Figure  \ref{fig1b} we compare the convergence of the Euler method and a Runge-Kutta of order 3; notice that with  $100$ time steps the RK method converges to a solution with an error of order $10^{-5}$ and by an estimation of the slope of the two lines we obtain 3 and 1.16, respectively for RK and Euler (as expected).
\begin{figure}[h!]
     \centering
     \begin{subfigure}{0.5\textwidth}
         \centering
         \includegraphics[width=1\linewidth]{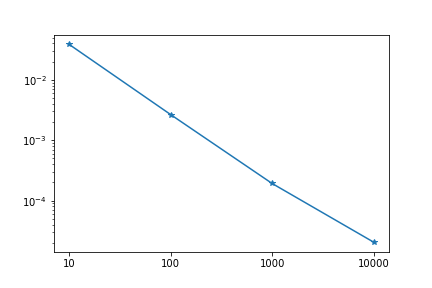}
         \caption{Linear convergence for an explicit Euler scheme.}
         \label{fig1a}
     \end{subfigure}
     \begin{subfigure}{0.5\textwidth}
         \centering
         \includegraphics[width=1\linewidth]{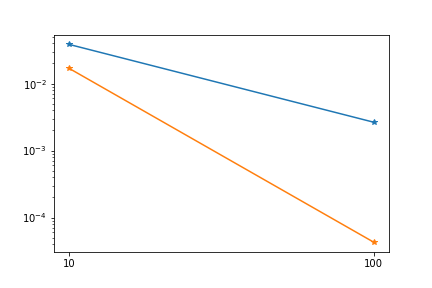}
         \caption{Comparison between explicit Euler (blue line) and an explicit Runge-Kutta (red line) of order 3}
         \label{fig1b}
     \end{subfigure}
     \caption{Order of convergence.}
     \label{fig:convergence}
\end{figure}

In Figure \ref{fig:fig1} we show the numerical result obtained with $\eta=0.006$, $h=1/1000$, $m=3$, the uniform measure on $[0,1]$ uniformily discretized with $100$ gridpoints and the pairwise interaction $w(x,y)=-\log(0.1+|x-y|)$. Notice that since we have developed our continuation method by  the entropic regularization of optimal transport, we can easily reconstruct the optimal (regularized) plan at each time $k$ by using the potential $\phi^{(k)}$.
 Moreover, it is interesting to notice that the optimal plan at each time step of the ODE stays deterministic (taking into account the blurring effect of the entropic regularization); that is, it is concentrated on a low dimensional structure.
\begin{remark}[Intermediate step of the ODE approach] It is interesting to highlight that each $ith$-step of the ODE approach actually returns the optimal solution to the multi-marginal problem with the cost $c_{kh}$. This, in particular, implies that one can also use this approach in order to retrieve the solutions of many multi-marginal problems with different costs, by choosing a suitable $c_\varepsilon$, instead of solving these problems individually using Sinkhorn.   
    
\end{remark}

 We are now interested in comparing the ODE approach and the Sinkhorn algorithm in terms of performance. In order to do this we consider the optimal solution of the regularized dual problem obtained with a gradient descent with backtracking and take it as the reference solution  to compute the relative error $\frac{\|\phi-\phi_{ref}\|_{\infty}}{\|\phi_{ref}\|_{\infty}}$.
Concerning the ODE, since we have already remarked above that it is smooth, we consider different high order methods such as 3rd, 5th and 8th order Runge-Kutta methods. By looking at Table \ref{tab1}, it is clear that all the  methods achieve almost the same relative error, but the number of iterations to reach it as well as the CPU time in seconds slightly differs. In particular we remark that a 3rd order RK is faster than a Sinkhorn in terms of computational time and iterations but less precise. The other RK methods achieve comparable precision to Sinkhorn with less iterations but the computational cost at the step of the ODE becomes now quite onerous  implying a  significant increase (especially 8th RK) in terms of CPU time.

\begin{figure}[h!]
\TabThree{
\includegraphics[width=.35\linewidth]{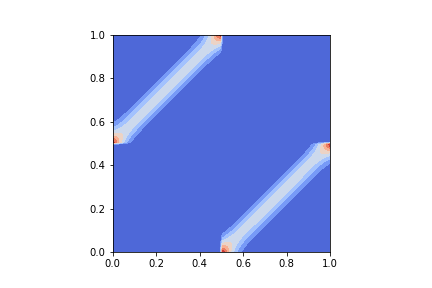}&\includegraphics[width=.35\linewidth]{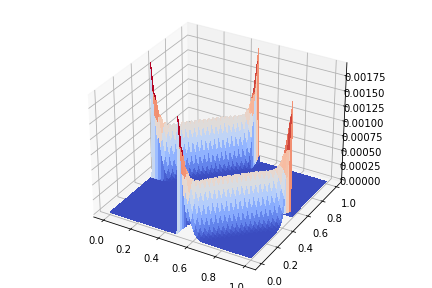} &\includegraphics[width=.35\linewidth]{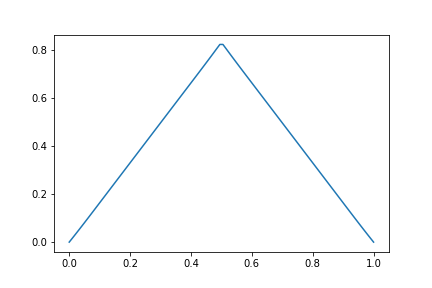}\\
\includegraphics[width=.35\linewidth]{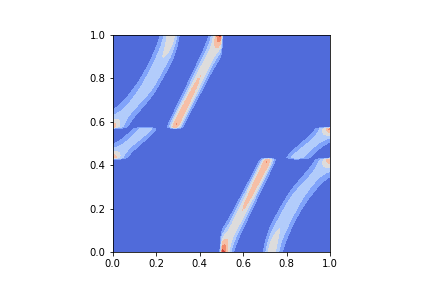}&\includegraphics[width=.35\linewidth]{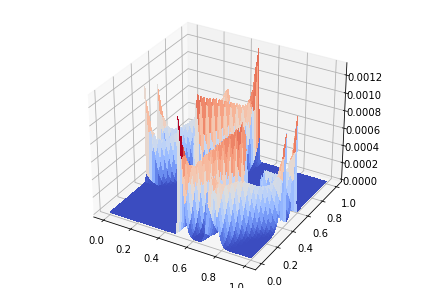} &\includegraphics[width=.35\linewidth]{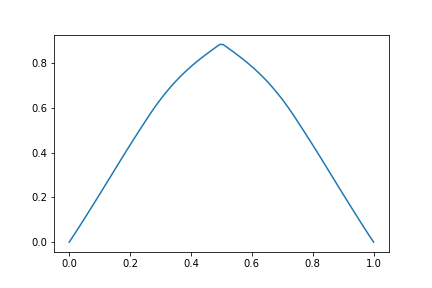}\\
\includegraphics[width=.35\linewidth]{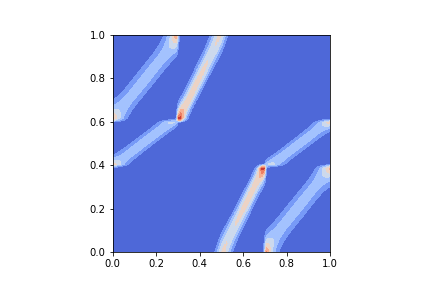}&\includegraphics[width=.35\linewidth]{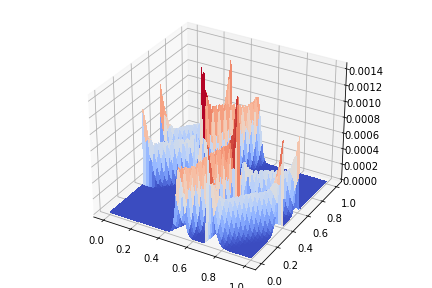} &\includegraphics[width=.35\linewidth]{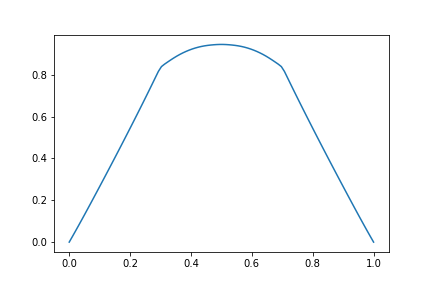}\\
\includegraphics[width=.35\linewidth]{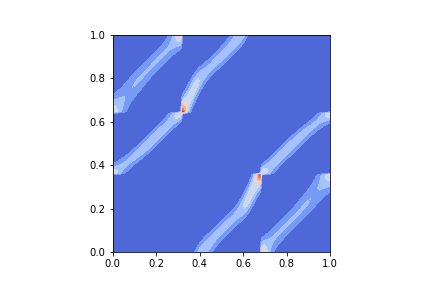}&\includegraphics[width=.35\linewidth]{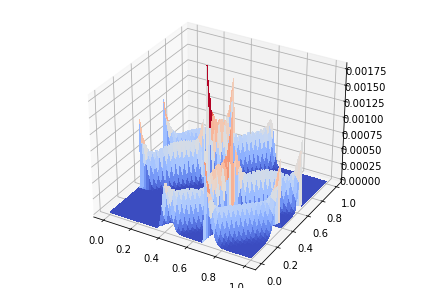} &\includegraphics[width=.35\linewidth]{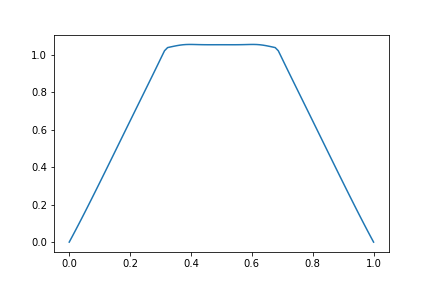}\\
\includegraphics[width=.35\linewidth]{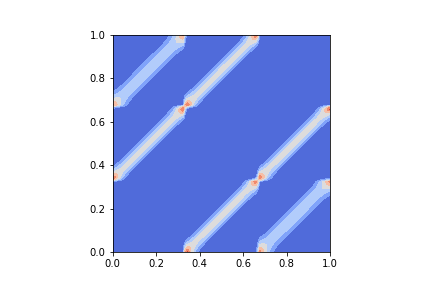}&\includegraphics[width=.35\linewidth]{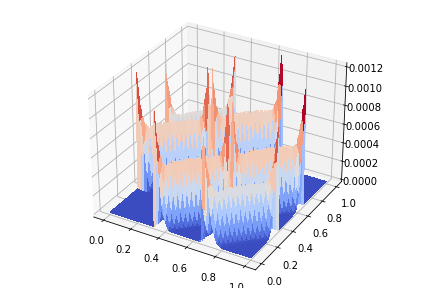} &\includegraphics[width=.35\linewidth]{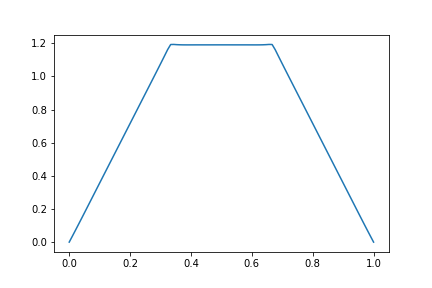}\\
(a)&(b)&(c)}
\caption{(Log cost) Column (a): support of  the coupling $\gamma^{\epsilon}_{1,2}$. Column (b): surface of the coupling $\gamma^{\epsilon}_{1,2}$. Column (c): potential $\phi(\epsilon)$. For $\epsilon=0,0.25,0.5,0.75,1$}
\label{fig:fig1}
\end{figure}

  
\begin{table}[h!]
\begin{tabular}{ |c|c|c|c|c| } 
 \hline
  & 3rd RK & 5th RK &8th RK & Sinkhorn \\ 
  \hline
relative error & $1.47\times 10^{-5}$ & $7.8\times 10^{-6}$   & $7.62\times 10^{-6}$ & $5.46\times 10^{-6}$ \\ 
\hline
iterations & 87 & 87 & 87 & 820 \\ 
 \hline
 CPU time (sec) &72.39 &158.9  & 385.1 & 102.8 \\
 \hline
\end{tabular}
\caption{Comparison between the ODE approach and Sinkhorn for the uniform density and 400 gridpoints}
\label{tab1}
\end{table}
In Figures \ref{fig:comp}-\ref{fig:fig3} we have kept the same data as before, but we have chose the negative harmonic cost, that is $w(x,y)=-|x-y|^2$. 
In particular in Figures \ref{fig1log} and \ref{fig1harm} we compare the potential obtained with Sinkhorn and the one at time $1$ of the ODE for the log cost and the negative harmonic, respectively.
We highlight that the solution at $\epsilon=0$ is $-Id$ and then the final coupling is supported, as expected, on the hyperplane $x+y+z=1.5$. 
 Looking at Figure \ref{fig:fig3}, notice that in this case the optimal transport plan is deterministic at the initial time, and also shortly after, but then it is possible to notice an interesting transition to  a spread transport plan.
\begin{figure}[h!]
     \centering
     \begin{subfigure}{0.4\textwidth}
         \centering
         \includegraphics[width=1\linewidth]{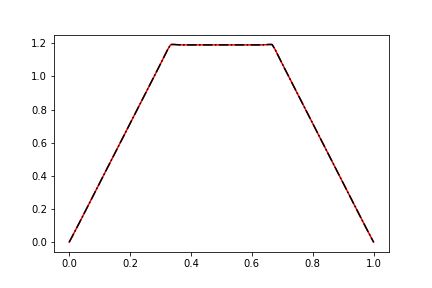}
         \caption{Log cost.}
         \label{fig1log}
     \end{subfigure}
     \begin{subfigure}{0.4\textwidth}
         \centering
         \includegraphics[width=1\linewidth]{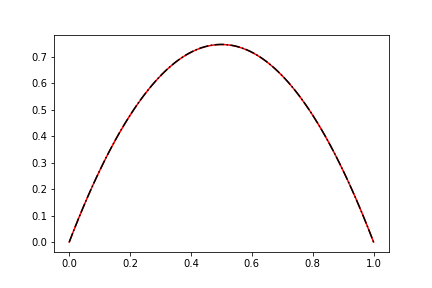}
         \caption{Negative harmonic cost.}
         \label{fig1harm}
     \end{subfigure}
     \caption{Optimal potential computed via Sinkhorn (red line). Potential computed via the ODE (black dot-dashed line).}
     \label{fig:comp}
\end{figure}


\begin{figure}[h!]
\TabThree{
\includegraphics[width=.35\linewidth]{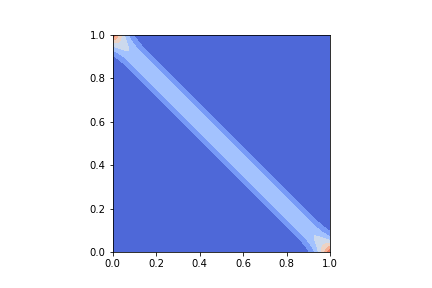}&\includegraphics[width=.35\linewidth]{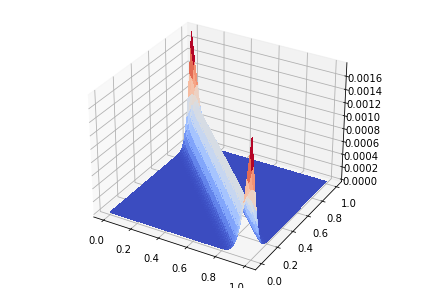} &\includegraphics[width=.35\linewidth]{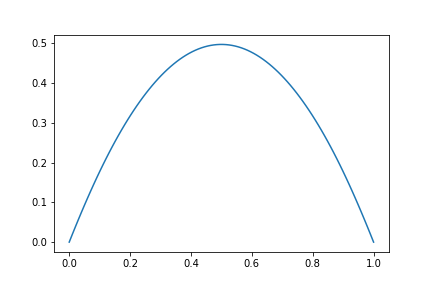}\\
\includegraphics[width=.35\linewidth]{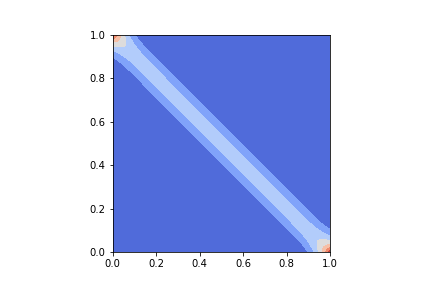}&\includegraphics[width=.35\linewidth]{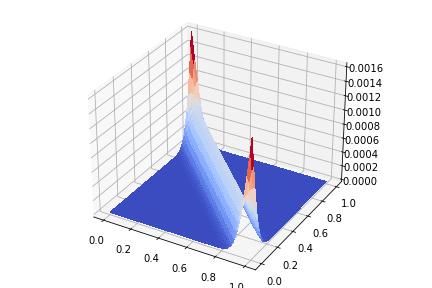} &\includegraphics[width=.35\linewidth]{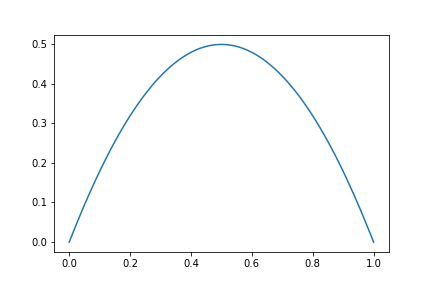}\\
\includegraphics[width=.35\linewidth]{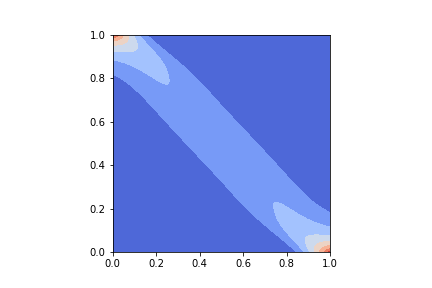}&\includegraphics[width=.35\linewidth]{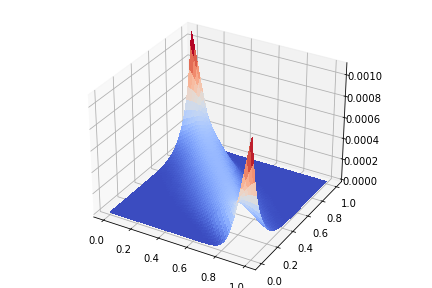} &\includegraphics[width=.35\linewidth]{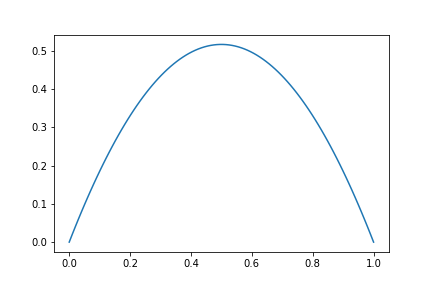}\\
\includegraphics[width=.35\linewidth]{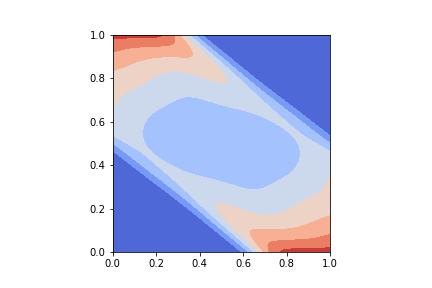}&\includegraphics[width=.35\linewidth]{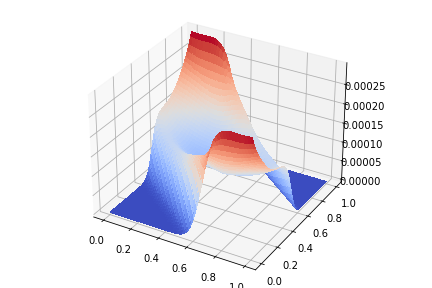} &\includegraphics[width=.35\linewidth]{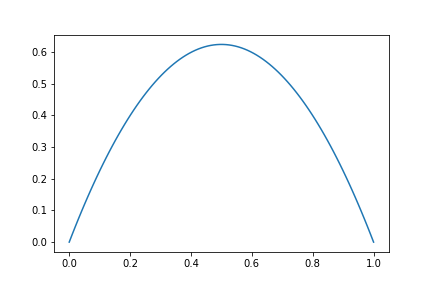}\\
\includegraphics[width=.35\linewidth]{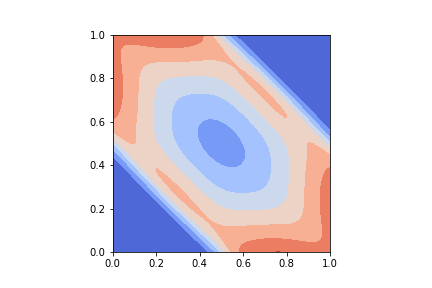}&\includegraphics[width=.35\linewidth]{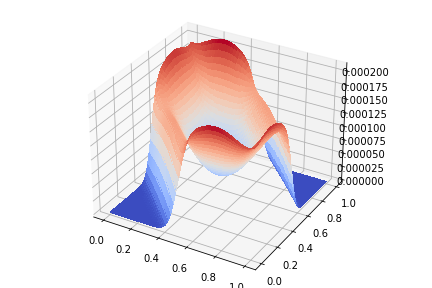} &\includegraphics[width=.35\linewidth]{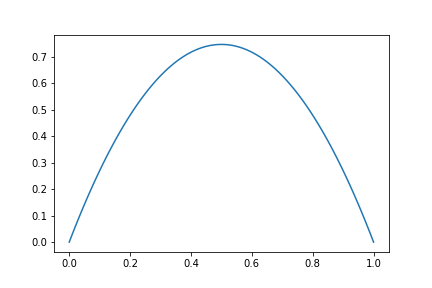}\\
(a)&(b)&(c)
}
\caption{(Negative Harmonic cost) Column (a): support of  the coupling $\gamma^{\epsilon}_{1,2}$. Column (b): surface of the coupling $\gamma^{\epsilon}_{1,2}$. Column (c): potential $\phi(\epsilon)$. For $\epsilon=0,0.25,0.5,0.75,1$}
\label{fig:fig3}
\end{figure}

\subsection{The non-symmetric ODE: Euler case}
\label{sec:euler}
In a series of papers \cite{brenier1989least,brenier1993dual,brenier1999minimal} Brenier proposed a relaxation of the incompressible Euler equation with constrained initial and final data interpreted as a geodesic on the group of measure preserving diffeomorphisms. Brenier's relaxed formulation consists in finding a probability measure over absolutely continuous paths which minimizes the average kinetic energy. In this framework the incompressibility is encoded by an additional constraint that at each time $t$, the distribution of position need be uniform.  If we consider a uniform discretization of $[0,T]$ (where $T$ is the final time) with $m$ steps in time, we recover a multi-marginal formulation of the Brenier principle with the specific cost function
\[c(x^1,\cdots,x^m)=\frac{m^2}{2T^2}\sum_{i=1}^{m-1}|x^{i+1}-x^i|^2\]
representing the discretized, in time, kinetic energy ($|\cdot|$ denotes the standard euclidean norm). The coupling $\gamma$ is the probability to find a \emph{generalized particle} on the discrete path $x^1,\cdots,x^m$. Because the fluid is incompressible, the $i$th marginal $\mu^i$ of $\gamma$ is the uniform on the $d-$dimensional cube $[0,1]^d$. For each $i\in\{1,\cdots,m\}$ the transition probability from time $1$ to time $i$ is given by the coupling 
\[\gamma_{1,i}\pi^i(x^1,x^i)=\bigg((x^1,\cdots,x^m)\mapsto(x^1,x^i)\bigg)_\#\gamma,\]
which represents the probability of finding a \emph{generalized particle} initially at $x^1$ to be at $x^i$ at time $i$. In order to impose the initial and final constraint we include by penalization, that is by adding a term to the cost function which now reads as
\[c(x^1,\cdots,x^m)=\frac{m^2}{2T^2}\sum_{i=1}^{m-1}|x^{i+1}-x^i|^2+\beta|F(x^1)-x^m|^2,\]
where $\beta>0$ is a penalization parameter in order to enforce the initial-final constraint. $F(x_1)$ represents the prescribed final position of the particle initially at position $x_1$, and the coupling $\gamma$ can be interpreted as a (generalization of) a discrete time geodesic between the identity mapping and $F$ on the space of measure preserving maps.

For an efficient implementation of Sinkhorn in this case we refer the reader to \cite{benamou2017generalized,thesislulu,benamouetalentropic}.
If we consider now the ODE setting, we have now to deal with a non symmetric case (the cost is not symmetric in the marginals anymore) and so to solve a system, still well posed, of ODEs. In particular we consider the following $c_\varepsilon$ cost
\[c_\varepsilon(x^1,\cdots,x^m)=\frac{m^2}{2T^2}|x^{2}-x^1|^2+\varepsilon\bigg(\frac{m^2}{2T^2}\sum_{i=2}^{m-1}|x^{i+1}-x^i|^2\bigg)+\beta|F(x^1)-x^m|^2.\]
For the numerical simulations we took $100$ gridpoints discretization of $[0,1]$, $m=9$ time marginals, $\beta=20$ and $\eta=0.002$. 
We solved the ODE system by using a $5th$ order  Runge-Kutta method and $h=1/100$. 
In figure \ref{fig:compPotEuler} we compare the potentials obtained via the ODE approach at $\varepsilon=1$ and Sinkhorn for two different initial-final configurations :   $F(x)=1-x$ figure \ref{figpotF1} and $F(x)=(x+1/2) \text{ mod } 1$ in figure \ref{figpotF2}. Notice in figure \ref{figpotF1} that for the particular choice of initial-final configuration the potentials $\phi_1$ and $\phi_m$ coincides as well as $\phi_i$ for $i=1,\cdots,m-1$.
In figure \ref{fig:euler} we plot the transition couplings for the initial-final configurations considered above, at $\varepsilon=1$. 
As for the symmetric case we compare the number of iteration and the CPU time of the ODE approach and Sinkhorn in table \ref{tab2}. Notice that Sinkhorn performs better than the $3rd$ order RK in terms of relative error (the reference solution for the error has been computed again via a gradient descent with backtracking). 
\begin{table}[h!]
\begin{tabular}{ |c|c|c|c|c| } 
 \hline
  & 3rd RK & 5th RK &8th RK & Sinkhorn \\ 
  \hline
relative error & $8.33\times 10^{-4}$ & $3.5\times 10^{-5}$   & $3.51\times 10^{-6}$ & $3.5\times 10^{-6}$ \\ 
\hline
iterations & 92 & 92 & 92 & 6649 \\ 
 \hline
 CPU time (sec) &321.58 &790.17  & 1225.45 & 327.82 \\
 \hline
\end{tabular}
\caption{Comparison between the ODE approach and Sinkhorn for the Euler case with $F(x)=(x+1/2)\text{ mod }1$, $100$ gridpoints and $9$ time marginals.}
\label{tab2}
\end{table}
\begin{figure}[h!]
     \centering
     \begin{subfigure}{0.5\textwidth}
         \centering
         \includegraphics[width=1.2\linewidth]{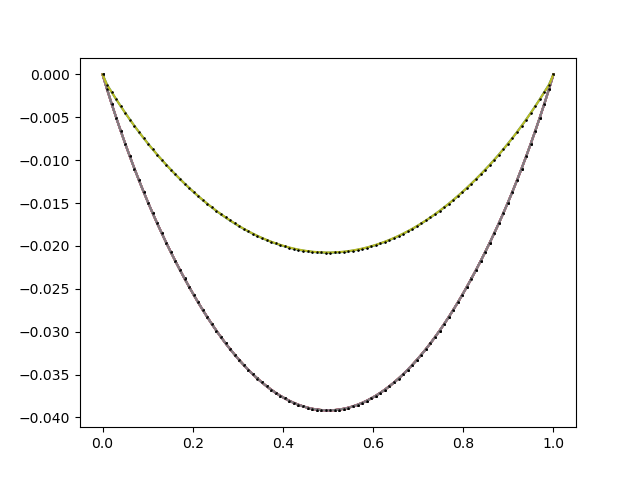}
         \caption{Comparison of potentials for $F(x)=1-x$.}
         \label{figpotF1}
     \end{subfigure}
     \begin{subfigure}{0.5\textwidth}
         \centering
         \includegraphics[width=1.2\linewidth]{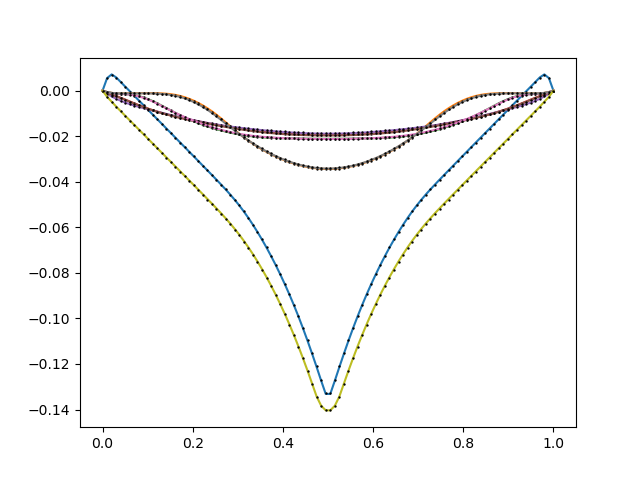}
         \caption{Comparison of potentials for $F(x)=x+1/2\text{ mod } 1$.}
         \label{figpotF2}
     \end{subfigure}
     \caption{Optimal potentials computed via Sinkhorn (black dots). Potential computed via the ODE (colored solid lines).}
     \label{fig:compPotEuler}
\end{figure}

\begin{figure}[h!]
\centering
\TabThree{
$t=0$&\includegraphics[width=.3\linewidth]{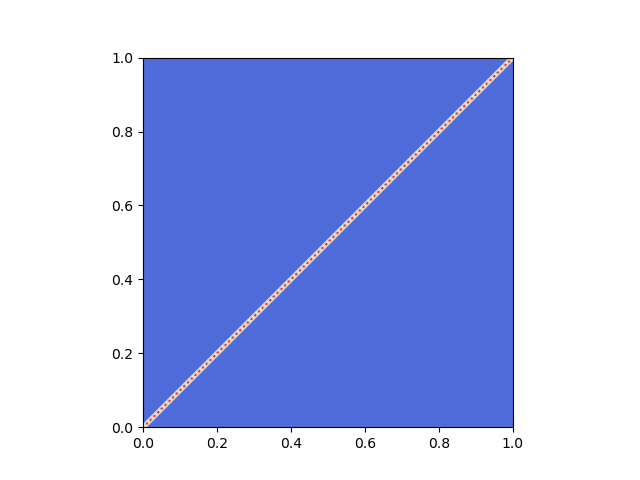}&\includegraphics[width=.3\linewidth]{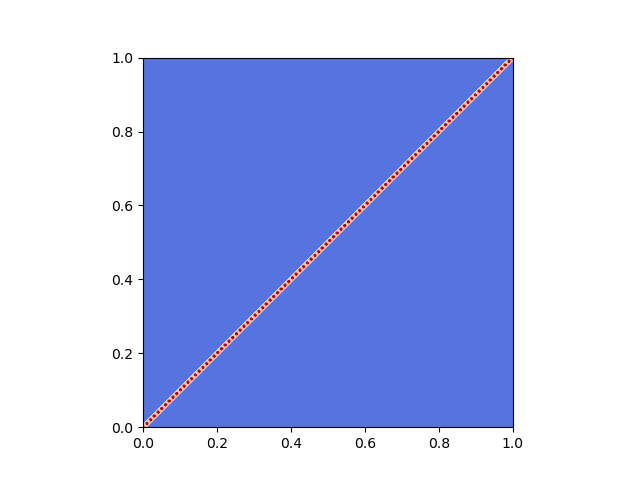}  \\

$t=1/8$&\includegraphics[width=.3\linewidth]{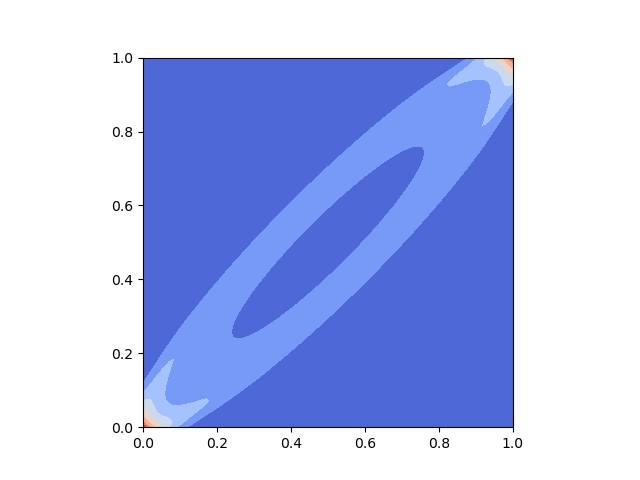}&\includegraphics[width=.3\linewidth]{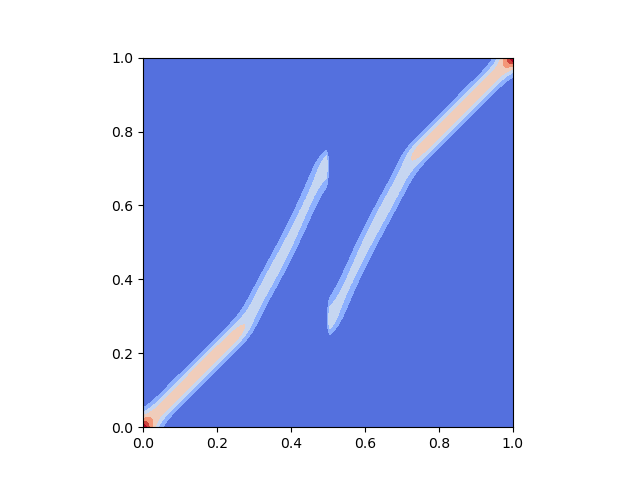}  \\
$t=1/4$&\includegraphics[width=.3\linewidth]{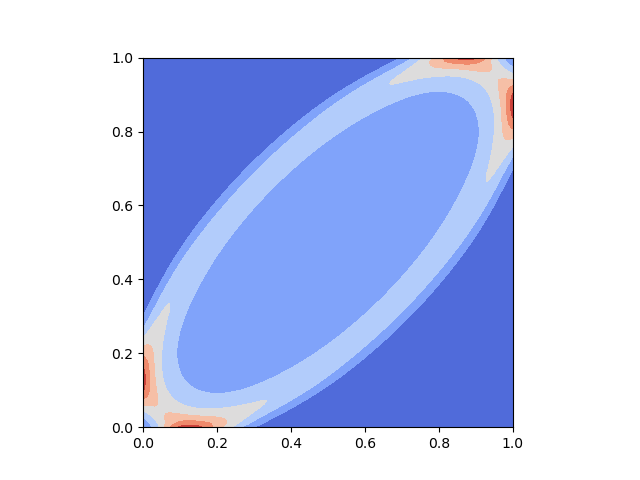}&\includegraphics[width=.3\linewidth]{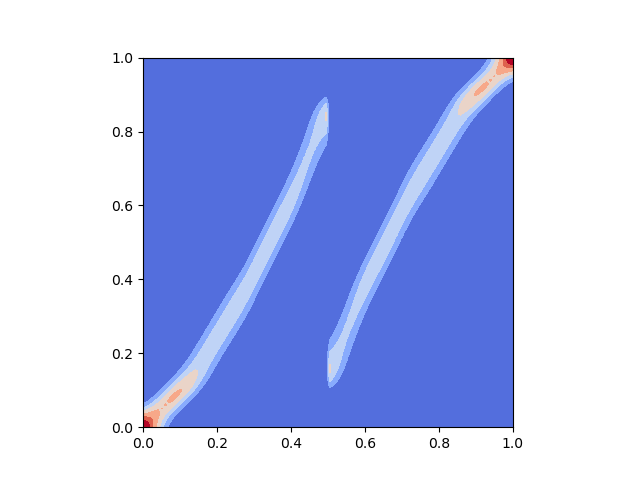}  \\
$t=1/2$&\includegraphics[width=.3\linewidth]{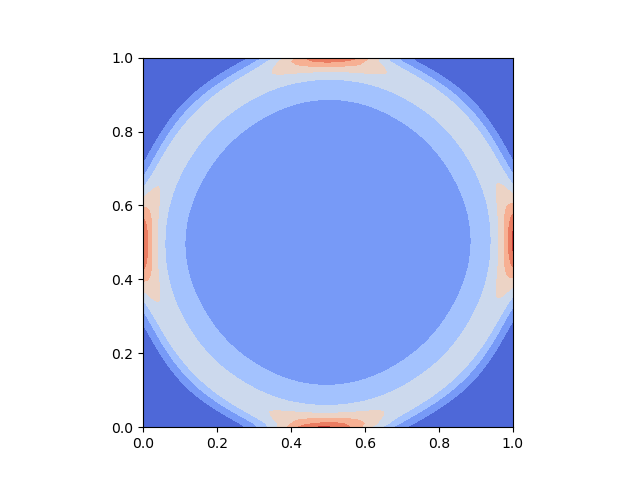}&\includegraphics[width=.3\linewidth]{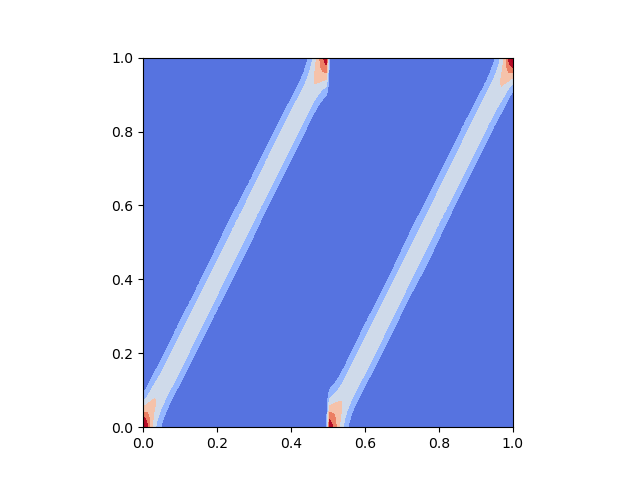}  \\
$t=3/4$&\includegraphics[width=.3\linewidth]{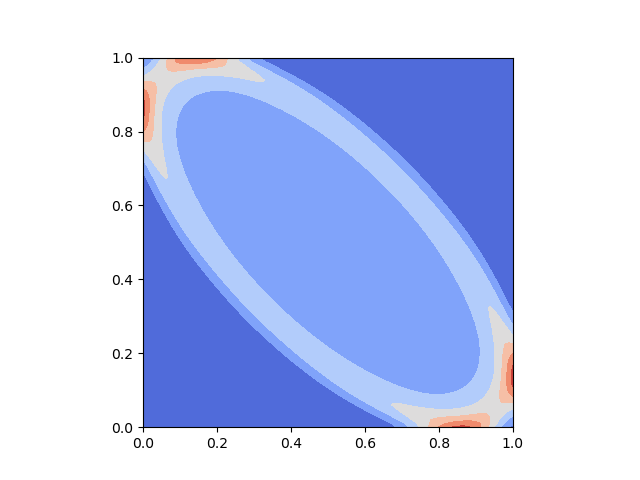}&\includegraphics[width=.3\linewidth]{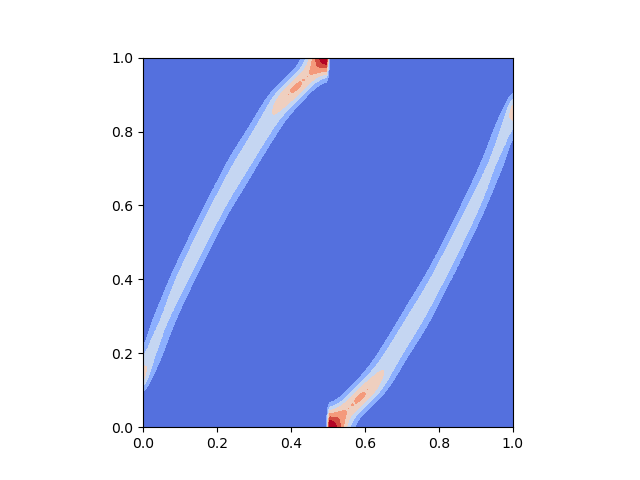}  \\

$t=7/8$&\includegraphics[width=.3\linewidth]{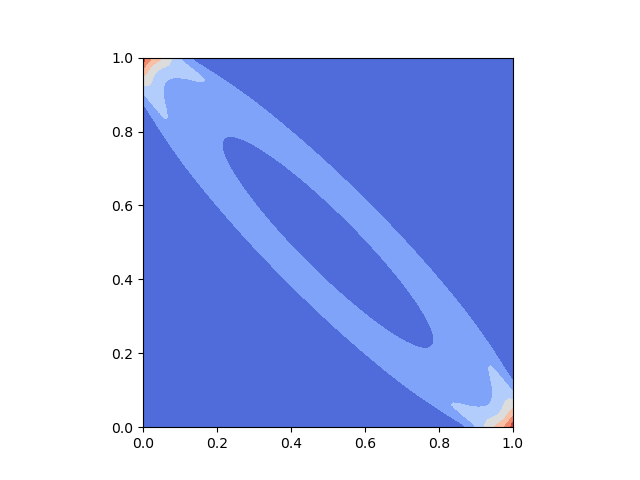}&\includegraphics[width=.3\linewidth]{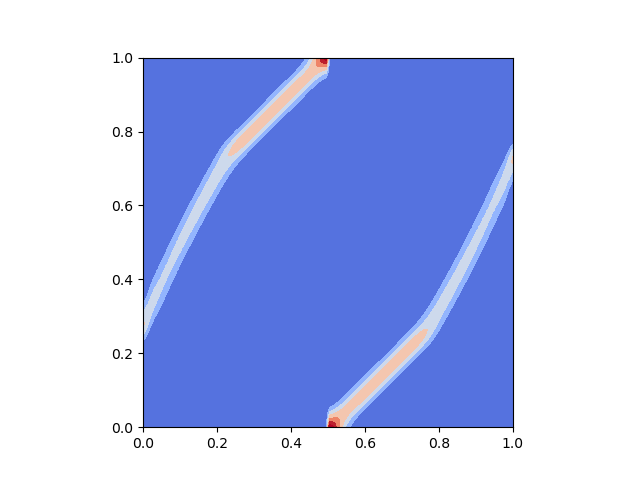}  \\

$t=1$& \includegraphics[width=.3\linewidth]{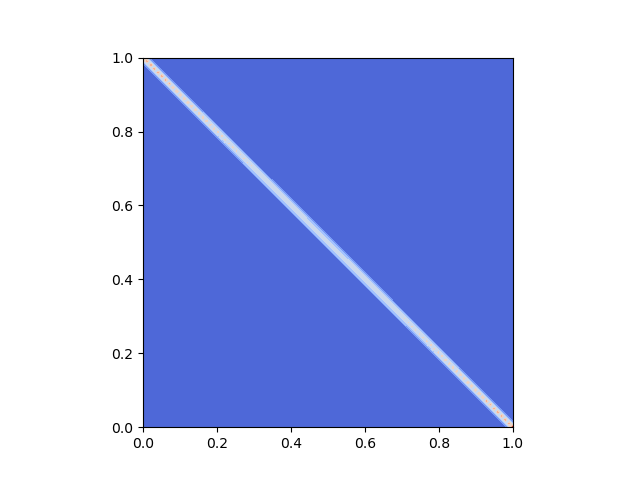}&\includegraphics[width=.3\linewidth]{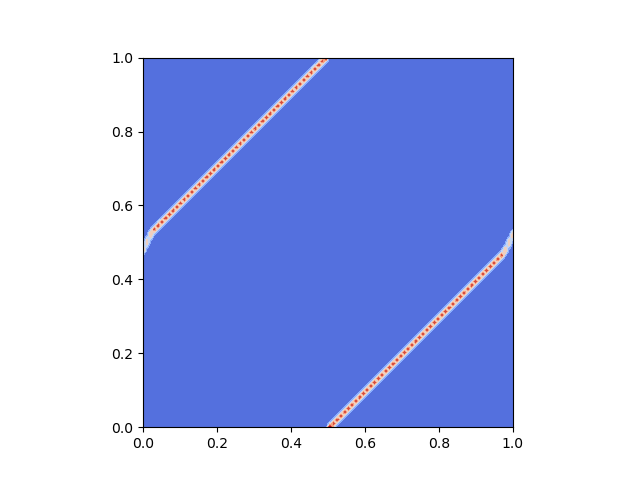}  \\

&$F(x)=1-x$ & $F(x)=x+1/2\text{ mod }1$ 
}
\caption{Transition plan $\gamma_{1,i}$ for different final configuration of Brenier relaxed formulation of incompressible Euler equations.}
\label{fig:euler}
\end{figure}

\bigskip

\noindent{\textbf{Acknowledgement.}} B.P. is pleased to acknowledge the support of National Sciences and Engineering Research Council of Canada Discovery Grant number  04658-2018.  He is also grateful for the kind hospitality at the Institut de Mathématiques d'Orsay, Université Paris-Saclay during his stay in November of 2019 as a missionaire scientifique invité, when this work was partially completed. L.N. was supported by a public grant as part of the Investissement d'avenir project, reference ANR-11-LABX-0056-LMH, LabEx LMH and  from the CNRS PEPS JCJC (2022). The authors would like to thank Guillaume Carlier for all the fruitful discussions which lead to Lemma \ref{lem:strong_convexity}, as well as two anonymous referees for their insightful comments. A preliminary version of the results in this paper appeared in the arxiv preprint version of our paper \cite{NennaPass22}, posted on January 3, 2022.  Prior to that paper being accepted for publication, these results were removed and they have since been refined and expanded to form the present manuscript. The authors would like to thank the anonymous referees for their comments and suggestions which helped to improve the paper.

%

\appendix

  \section{Second derivatives of $\tilde\Phi$}
 \label{app:seconderv}
 In this appendix we detail the second derivatives of $\tilde\Phi$ with respect to $\phi$ and $\epsilon$.
 Let us consider firstly the term \eqref{eq:ODE}
 \[ \dfrac{\partial}{\partial \epsilon}\nabla_\varphi\Phi(\varphi,\epsilon),\]
notice that  $\nabla_\varphi\Phi(\varphi,\epsilon)$ it is a composition of an exponential with a linear function in $\epsilon$, meaning that it is differentiable with respect to $\epsilon$.
We obtain then the following
 \[ \frac{\partial}{\partial \epsilon}(\nabla_\phi\Phi(\phi,\epsilon))_{\phi_z}=-\dfrac{\exp(\phi_z/\eta)}{\eta}\sum_{\bar x\in X^{m-1}}\partial_\epsilon(c_\epsilon(z,\bar x))\exp\Bigg(\frac{\sum_{i=1}^{m-1}\phi_{x^i}-c_\epsilon(z,\bar x)}{\eta}\Bigg),\]
\[\begin{split}
  \frac{\partial}{\partial \epsilon} &\nabla_{\phi}\Phi(\phi(\epsilon),\epsilon):=-\frac{1}{\eta}\exp(\phi/\eta)\rho\Bigg(\sum_{(x,\bar y)\in X^{m-1}}\partial_\epsilon c_\epsilon(x,z,\bar y)\exp\Big(\frac{\sum_{i=3}^{m}\phi_{y^i}- c_\epsilon(x,z,\bar y)}{\eta}\Big)\tilde \rho_{\bar y}\bar\rho_x\\
  &+\sum_{(x,\bar y)\in X^{m-1}}\exp\Big(\frac{\sum_{i=3}^{m}\phi_{y^i}- c_\epsilon(x,z,\bar y)}{\eta}\Big)\tilde \rho_{\bar y}\frac{\sum_{\bar w\in X^{m-1}}\partial_\epsilon c_\epsilon(x,\bar w)\exp\Big(\frac{\sum_{i=2}^{m}\phi_{w^i}- c_\epsilon(x,\bar w)}{\eta}\Big)\tilde \rho_{\bar w}}{\sum_{\bar w\in X^{m-1}}\exp\Big(\frac{\sum_{i=2}^{m}\phi_{w^i}- c_\epsilon(x,\bar w)}{\eta}\Big)\tilde \rho_{\bar w}}\bar\rho_x \Bigg),
\end{split} \]
 where
 \[\partial_\epsilon c_\epsilon(z,\bar x)=\sum_{i,j=2,i\neq j}^{m-1}w(x^i,x^j). \]
 Concerning the second derivative with respect to $\phi$. it is again quite easy to see that $\Phi$ is twice differentiable, then we have 
\[ D^2_{\phi,\phi}\tilde\Phi=\frac{1}{\eta}\diag(I_1)+\frac{m-2}{\eta}(e^{\phi/\eta}\rho)\otimes (e^{\phi/\eta}\rho)I_2-\frac{m-1}{\eta}\sum_{y\in X}I_3^y\otimes I^y_3\rho_y, \]
where
\[ 
\begin{split}
&(I_1)_z=e^{\phi_z/\eta}\rho_z\sum_y\sum_{\bar x\in X^{m-2}}\exp{\Bigg(\frac{\sum_{i=3}^{m}\phi_{x^i}-c_\epsilon(y,z,\bar x)}{\eta}\Bigg)}(\otimes^{m-2}\rho)_{\bar x}\bar \rho_y,\\
&(I_2)_{z,w}= \sum_y\sum_{\bar x\in X^{m-3}}\exp{\Bigg(\frac{\sum_{i=4}^{m}\phi_{x^i}-c_\epsilon(y,z,w,\bar x)}{\eta}\Bigg)}(\otimes^{m-3}\rho)_{\bar x}\bar \rho_y,\\
&(I^y_3)_z=\dfrac{e^{\phi_z/\eta}\rho_z\sum_{\bar x\in X^{m-2}}\exp{\Bigg(\frac{\sum_{i=3}^{m}\phi_{x^i}-c_\epsilon(y,z,\bar x)}{\eta}\Bigg)}(\otimes^{m-2}\rho)_{\bar x} }{\sum_{\bar x\in X^{m-1}}\exp\Bigg(\frac{\sum_{i=2}^{m}\phi_{x^i}-c_\epsilon(y,\bar x)}{\eta}\Bigg)\tilde\rho_{\bar x}}\
\end{split}
\]

\bibliographystyle{plain}
\bibliography{bibli}

\end{document}